\documentclass{amsart}

\usepackage{hyperref}
\usepackage{amsmath,amsthm,amssymb}

\newtheorem{thm}{Theorem}
\newtheorem{prop}[thm]{Proposition}
\newtheorem{lem}[thm]{Lemma}
\newtheorem{cor}[thm]{Corollary}
\newtheorem{clm}[thm]{Claim}
\newtheorem{sbclm}[thm]{Subclaim}

\theoremstyle{definition}
\newtheorem{dfn}[thm]{Definition}
\newtheorem{rem}[thm]{Remark}
\newtheorem{set}[thm]{Setting}
\newtheorem{prob}[thm]{Problem}

\numberwithin{thm}{section}
\numberwithin{equation}{section}

\DeclareMathOperator{\diam}{diam}
\DeclareMathOperator{\vol}{vol}
\DeclareMathOperator{\dom}{Dom}

\begin{document}

\title[A fibration theorem for collapsing sequences]{A fibration theorem for collapsing sequences of Alexandrov spaces}
\author[T. Fujioka]{Tadashi Fujioka}
\address{Department of Mathematics, Kyoto University, Kitashirakawa, Kyoto 606-8502, Japan}
\email{\href{mailto:tfujioka@math.kyoto-u.ac.jp}{tfujioka@math.kyoto-u.ac.jp}}
\date{\today}
\subjclass[2010]{53C20, 53C23}
\keywords{Alexandrov spaces, collapse, fibration, strainers, almost regular maps}

\begin{abstract}
Suppose a sequence $M_j$ of Alexandrov spaces collapses to a space $X$ with only weak singularities.
Yamaguchi constructed a map $f_j:M_j\to X$ called an almost Lipschitz submersion for large $j$.
We prove that if $M_j$ has a uniform positive lower bound for the volumes of spaces of directions, which is sufficiently large compared to the weakness of singularities of $X$, then $f_j$ is a locally trivial fibration.
Moreover, we show some properties on the intrinsic metric and the volume of the fibers of $f_j$.
\end{abstract}

\maketitle

\section{Introduction}\label{sec:intro}

Let $M_j$ be a sequence of $n$-dimensional Alexandrov spaces with curvature $\ge\kappa$ and diameter $\le D$.
It is well-known that $M_j$ has a convergent subsequence in the Gromov-Hausdorff distance and that the limit space $X$ is also an Alexandrov space of dimension $\le n$ and with curvature $\ge\kappa$.
The main problem of the convergence theory of Alexandrov spaces is to determine the relation between the topology and geometry of $X$ and those of $M_j$ with large $j$.

According to Perelman's stability theorem, if $\dim X=n$, then $M_j$ is homeomorphic to $X$ (\cite{P1}, \cite{K}).
The case $\dim X<n$ is called a collapse.
In this case, Yamaguchi \cite{Y1} proved that if both $M_j$ and $X$ are Riemannian manifolds, then there exists a locally trivial fibration $f_j:M_j\to X$, which is an almost Riemannian submersion.
This result can be generalized to the case when both $M_j$ and $X$ have only weak singularities (\cite[9.13]{BGP}).
Furthermore, Yamaguchi \cite{Y2} also proved that even if $M_j$ is a general Alexandrov space, if $X$ has only weak singularities, then there exists a map $f_j:M_j\to X$ called an almost Lipschitz submersion, which is a generalization of an almost Riemannian submersion.

To state Yamaguchi's almost Lipschitz submersion theorem, we introduce some notation.
We denote by $\delta$ a positive number less than some constant depending only on $n$ and $\kappa$, and by $\varkappa(\delta)$ a positive function depending only on $n$ and $\kappa$ such that $\varkappa(\delta)\to0$ as $\delta\to0$.
For a positive number $a$, we denote by $c(a)$ a positive constant depending only on $n$, $\kappa$ and $a$, which is much smaller than $a$.

\begin{thm}[{\cite[0.2]{Y2}}]\label{thm:subm}
Let $k<n$ and let $X$ be a $k$-dimensional Alexandrov space with curvature $\ge\kappa$ such that every point has a $(k,\delta)$-strainer with length $>\ell$.
Let $M$ be an $n$-dimensional Alexandrov space with curvature $\ge\kappa$ that is $\mu$-close to $X$ in the Gromov-Hausdorff distance.
Suppose $\mu<c(\ell\delta^2)\ll\ell\delta^2$.
Then, there exists a $\varkappa(\delta)$-almost Lipschitz submersion $f:M\to X$ in the following sense:
\[\left|\frac{|f(x)f(y)|}{|xy|}-\sin\inf_z\angle yxz\right|<\varkappa(\delta)\]
for any $x,y\in M$, where the infimum is taken over all $z\in f^{-1}(f(x))$.
\end{thm}

Yamaguchi conjectured that the map $f$ is actually a locally trivial fibration.
Rong and Xu \cite{RX} showed that it is true if each fiber of $f$ is a topological manifold (without boundary) of codimension $k$.
Xu \cite{X} (cf.\ \cite{XY}) also proved that $f$ is a Hurewicz fibration.
Their proofs are based on the construction of neighborhood retractions to the fibers, which can be applied to a wider class of maps called $e^\varepsilon$-Lipschitz and co-Lipschitz maps.
We prove Yamaguchi's conjecture from a different point of view, in the following case:

\begin{thm}\label{thm:main}
Under the conditions of Theorem \ref{thm:subm}, let $\varepsilon$ be a lower bound for the volume of the space of directions at each point of $M$.
Suppose in addition $\delta<c(\varepsilon)\ll\varepsilon$.
Then, the map $f$ is a locally trivial fibration.
\end{thm}

In the case of collapse of codimension one, the additional assumption is always satisfied.

\begin{cor}\label{cor:main}
Under the conditions of Theorem \ref{thm:subm}, if $k=n-1$, then $f$ is a locally trivial fibration whose fiber is homeomorphic to a circle or a closed interval.
\end{cor}

Let us explain how to prove Theorem \ref{thm:main}.
Let $\{(a_i,b_i)\}_{i=1}^k$ be a $(k,\delta)$-strainer at $p\in X$, where $k=\dim X$.
Then, the distance map $\varphi=(|a_1\cdot|,\dots,|a_k\cdot|):X\to\mathbb R^k$ gives a local chart around $p$.
Let $\hat\varphi:M\to\mathbb R^k$ and $\hat p\in M$ be natural lifts of $\varphi$ and $p$, respectively.
Then, $\varphi^{-1}\circ\hat\varphi$ gives a local map between neighborhoods of $\hat p$ and $p$.
The global map $f$ in Theorem \ref{thm:subm} is constructed by gluing such local maps.
We prove that $f$ satisfies
\begin{equation}\label{eq:main}
\bigl|(\varphi\circ f(x)-\varphi\circ f(y))-(\hat\varphi(x)-\hat\varphi(y))\bigr|<\varkappa(\delta)|xy|
\end{equation}
for any $x$, $y$ near $\hat p$.
Roughly speaking, the differential of $\varphi\circ f$ is close to that of $\hat\varphi$.
This inequality allows $f$ to inherit the properties of $\varphi^{-1}\circ\hat\varphi$.

In \cite{P1}, Perelman developed the theory of noncritical maps and proved that a proper noncritical map is a locally trivial fibration (here, the definition of noncriticality includes an assumption on the volume of spaces of directions).
In particular, the above $\hat\varphi$ is a locally trivial fibration near $\hat p$ under the additional assumption of Theorem \ref{thm:main}.
We slightly modify the definition of noncritical maps in terms of the inequality \eqref{eq:main} and show that $\varphi\circ f$ is also a locally trivial fibration.
From this point of view, the general case of Yamaguchi's conjecture is reduced to the problem of proving Perelman's fibration theorem for noncritical maps without the assumption on the volume of spaces of directions (see also Remark \ref{rem:mor}).

Next, we discuss the fibers of the map $f$ in Theorem \ref{thm:subm} (not in Theorem \ref{thm:main}).
Fibers of almost regular maps such as $\hat\varphi$ were studied in \cite[\S11--12]{BGP}.
The inequality \eqref{eq:main} enables us to apply the arguments there to $\varphi\circ f$.
It is also known that the fundamental group of the homotopy fiber of $f$ contains a nilpotent subgroup whose index is uniformly bounded above (\cite{FY}, \cite{Y2}, \cite{KPT}, \cite{X}, \cite{XY}).
Here, we show some metric properties of the fibers of $f$.
Note that the diameters of the fibers of $f$ are very small (less than a constant multiple of the Gromov-Hausdorff distance $\mu$ between $M$ and $X$).
Let $\vol_m$ denote the $m$-dimensional Hausdorff measure.

\begin{thm}\label{thm:fbr}
Let $f:M\to X$ be the map of Theorem \ref{thm:subm}.
Let $F_p$ denote the fiber $f^{-1}(p)$ over $p\in X$.
\begin{enumerate}
\item The induced intrinsic metric of $F_p$ is almost isometric to the original one, that is,
\[|xy|_{F_p}<(1+\varkappa(\delta))|xy|\]
for any $x,y\in F_p$, where $|\ ,\ |_{F_p}$ denotes the induced intrinsic metric of $F_p$.
(In particular, we can use both metrics below.)
\item The Hausdorff dimension of $F_p$ is $n-k$.
Moreover, we have
\[0<\vol_{n-k}F_p<C(\diam F_p)^{n-k},\]
where $C$ is a positive constant depending only on $n$ and $\kappa$.
\item Fix $p\in X$.
Then, for any $q\in X$ sufficiently close to $p$, we have
\[\left|\frac{\vol_{n-k}F_q}{\vol_{n-k}F_p}-1\right|<\varkappa(\delta).\]
\end{enumerate}
\end{thm}

The property (1) was known for fibers of almost regular maps such as $\hat\varphi$ (\cite[11.11]{BGP}).
Thus, it also holds for the fibers of $\varphi\circ f$ satisfying the inequality \eqref{eq:main}.
This result was stated in the first version of \cite{X}, but has been deleted in the second version (due to an oversight as explained in the abstract on arXiv).
Regarding (2), it was only known that the topological dimension of fibers of almost regular maps such as $\hat\varphi$ is no greater than $n-k$ (\cite[11.8]{BGP}).
The left inequality in (2) was conjectured in \cite[4.2]{MY}.
The author does not know whether the volume of $F_p$ is actually continuous in (3).

\begin{rem}\label{rem:main}
Strictly speaking, we should not assert that our results (Theorem \ref{thm:main}, Corollary \ref{cor:main} and Theorem \ref{thm:fbr}) hold for the almost Lipschitz submersion constructed by Yamaguchi \cite{Y2}.
In fact, we construct the map $f$ of Theorem \ref{thm:subm} in a different way from \cite{Y2} and show the inequality \eqref{eq:main} for this new map but not for the original one.
However, both constructions are essentially the same and \eqref{eq:main} actually holds for Yamaguchi's map.
See Theorem \ref{thm:con} and Remark \ref{rem:con2}.
\end{rem}

The organization of this paper is as follows:
In \S\ref{sec:note}, we introduce some notation and conventions which will be used throughout this paper.
In \S\ref{sec:pre}, we recall some basic facts on Alexandrov spaces, especially strainers, and list a few lemmas from \cite{P1} for later use.
In \S\ref{sec:con}, we construct the map $f$ of Theorem \ref{thm:subm} and show the inequality \eqref{eq:main}.
In \S\ref{sec:per}, we generalize Perelman's fibration theorem for noncritical maps in terms of the inequality \eqref{eq:main} and prove Theorem \ref{thm:main} and Corollary \ref{cor:main}.
Most of the contents of this section are slight modifications of those of \cite[\S3]{P1}.
In \S\ref{sec:fbr}, we define the notion of almost regular maps in terms of the inequality \eqref{eq:main} and study fibers of them (note that our definition of almost regular maps is different from that of \cite{BGP}).
In \S\ref{sec:int}, we prove Theorem \ref{thm:fbr}(1).
In \S\ref{sec:low}, we prove the left inequality of Theorem \ref{thm:fbr}(2), and in \S\ref{sec:up}, we prove the right inequality.
In \S\ref{sec:conti}, we prove Theorem \ref{thm:fbr}(3).

\section{Notation and conventions}\label{sec:note}

The dimension of Alexandrov spaces is usually denoted by $n$.
The lower curvature bound $\kappa$ of Alexandrov spaces is fixed and omitted unless otherwise stated.
A positive integer $k$ is usually no greater than $n$ and is often less than $n$.
We always assume that a lower bound $\ell$ for the lengths of strainers is no greater than $1$ especially when $\kappa<0$ (indeed, all our arguments using strainers are local).

We denote by $c$ and $C$ various small and large positive constants, respectively.
Unless otherwise stated, such constants depend only on $n$ and $\kappa$.
If they depend on additional parameters, it will be indicated explicitly, like $c(\varepsilon)$.

We always assume that a positive number $\delta$ is smaller than some constant $c_0$ depending only on $n$ and $\kappa$.
We denote by $\varkappa(\delta)$ various positive functions such that $\varkappa(\delta)\to0$ as $\delta\to0$.
Unless otherwise stated, $\varkappa$ depends only on $n$ and $\kappa$.
In this case, we often assume that $\varkappa(\delta)$ is also smaller than $c_0$.

In \S\ref{sec:lem}, \S\ref{sec:per} and \S\ref{sec:up}, we use another positive number $\varepsilon$.
In these sections, $\varkappa$ may depend additionally on $\varepsilon$.
Furthermore, we assume that $\delta$ is smaller than some constant $c_0(\varepsilon)$ depending only on $n$, $\kappa$ and $\varepsilon$, which is much smaller than any other $c(\varepsilon)$ appearing in these sections.
We often assume that $\varkappa(\delta)$ is also smaller than $c_0(\varepsilon)$.

The $m$-dimensional Hausdorff measure is denoted by $\vol_m$.
We use the standard Euclidean norm on $\mathbb R^k$ except in \S\ref{sec:per}, where we use the maximum norm instead.

\section{Preliminaries}\label{sec:pre}

We first recall some basic facts on Alexandrov spaces.
See \cite{BGP} or \cite{BBI} for more details.

Let $M$ be an $n$-dimensional Alexandrov space with curvature $\ge\kappa$.
For a geodesic triangle $\triangle pqr$ in $M$ with vertices $p$, $q$ and $r$, we denote by $\tilde\triangle pqr$ a geodesic triangle with the same sidelengths in the simply-connected complete surface of constant curvature $\kappa$.
Then, by the definition of an Alexandrov space, the natural correspondence from $\triangle pqr$ to $\tilde\triangle pqr$ is nonexpanding.
Let $\angle qpr$ denote the angle of $\triangle pqr$ at $p$ and $\tilde\angle qpr$ the corresponding angle of $\tilde\triangle pqr$.
Then, the Alexandrov convexity implies that $\angle qpr\ge\tilde\angle qpr$.

For $p\in M$, we denote by $\Sigma_p$ the space of directions at $p$.
Then, $\Sigma_p$ is an $(n-1)$-dimensional compact Alexandrov space with curvature $\ge1$.
For a point $q\in M$, we denote by $q'_p\in\Sigma_p$ one of the directions of shortest paths from $p$ to $q$.
Furthermore, for a closed subset $A\subset M$, we denote by $A'_p\subset\Sigma_p$ the set of all directions of shortest paths from $p$ to $A$.
We sometimes use the notation $Q'_p\subset\Sigma_p$ to denote the set of all directions of shortest paths from $p$ to $q$ by regarding $Q$ as $\{q\}$.

The class of all Alexandrov spaces with dimension $\le n$, curvature $\ge\kappa$ and diameter $\le D$ is compact with respect to the Gromov-Hausdorff distance.
Furthermore, the class of all pointed Alexandrov spaces with dimension $\le n$ and curvature $\ge\kappa$ is compact with respect to the pointed Gromov-Hausdorff topology.

\subsection{Strainers}\label{sec:str}

Let $M$ be an $n$-dimensional Alexandrov space.

\begin{dfn}\label{dfn:str}
A point $p\in M$ is said to be \textit{$(k,\delta)$-strained} if there exists $k$ pairs $\{(a_i,b_i)\}_{i=1}^k$ of points in $M$ such that
\begin{gather*}
\tilde\angle a_ipb_i>\pi-\delta,\\
\tilde\angle a_ipa_{i'},\tilde\angle a_ipb_{i'},\tilde\angle b_ipb_{i'}>\pi/2-\delta
\end{gather*}
for all $1\le i\neq i'\le k$.
The collection $\{(a_i,b_i)\}_{i=1}^k$ is called a \textit{$(k,\delta)$-strainer} at $p$.
The number $\min_i\{|a_ip|,|b_ip|\}$ is called the \textit{length} of this strainer.
\end{dfn}

We now describe a few basic properties of strainers and strained points.

\begin{lem}\label{lem:str}
Let $(a,b)$ be a $(1,\delta)$-strainer at $p\in M$ with length $>\ell$.
Then, we have
\[\left|\tilde\angle axy-\angle axy\right|<\varkappa(\delta)\]
for any $x,y\in B(p,\ell\delta)$ and any shortest paths $xa$, $xy$.
\end{lem}

See \cite[5.6]{BGP} or \cite[10.8.13]{BBI} for the proof.

Let $S^k(\Sigma)$ be the $k$-fold spherical suspension over a space $\Sigma$ of curvature $\ge1$ (see \cite[4.3.1]{BGP} for the definition).
Note that it is isometric to the spherical join of $\Sigma$ and the unit sphere $\mathbb S^{k-1}$ of dimension $k-1$.
Thus, $S^k(\Sigma)$ contains an isometric copy of $\mathbb S^{k-1}$ (we identify them).
Let $\{(\xi_i,\eta_i)\}_{i=1}^k$ be a collection of pairs of points in $\mathbb S^{k-1}\subset S^k(\Sigma)$ such that
\[|\xi_i\eta_i|=\pi,\quad|\xi_i\xi_{i'}|=|\xi_i\eta_{i'}|=|\eta_i\eta_{i'}|=\pi/2\]
for all $1\le i\neq i'\le k$.
We call such a collection an \textit{orthogonal $k$-frame} of $S^k(\Sigma)$.
Conversely, if a space of curvature $\ge1$ has such a collection, then it is isometric to a $k$-fold spherical suspension (see \cite[10.4.3]{BBI}).

The space of directions at a strained point is close to a suspension in the following sense:

\begin{lem}\label{lem:susp}
Let $p\in M$ be a $(k,\delta)$-strained point with a strainer $\{(a_i,b_i)\}_{i=1}^k$.
Let $(A_i)'_p$, $(B_i)'_p$ denote the set of all directions of shortest paths from $p$ to $a_i$, $b_i$, respectively.
Then, there exists a $\varkappa(\delta)$-approximation $\Sigma_p\to S^k(\Sigma)$ which sends $\{((A_i)'_p,(B_i)'_p)\}_{i=1}^k$ to an orthogonal $k$-frame of $S^k(\Sigma)$, where $\Sigma$ is a space of curvature $\ge1$ and dimension $\le n-k-1$ (possibly empty).
In particular, if $k=n$, then $\Sigma_p$ is $\varkappa(\delta)$-close to $\mathbb S^{n-1}$ in the Gromov-Hausdorff distance.
\end{lem}

The proof is by contradiction (see \cite[3.2]{F} for instance).

\begin{rem}\label{rem:susp}
Furthermore, if $k<n$, then $\Sigma$ is nonempty.
Indeed, $\mathbb S^{k-1}$ cannot be a limit of a collapsing sequence of Alexandrov spaces of curvature $\ge1$.
This follows, for instance, from the fact that such a limit space either has diameter no greater than $\pi/2$ or contains a proper extremal subset (\cite[3.2]{P3}).
In particular, if $k=n-1$, then $\Sigma_p$ is $\varkappa(\delta)$-close to $\mathbb S^{n-1}$ or the closed unit hemisphere $\mathbb S^{n-1}_+$.
\end{rem}

Let $0<\varepsilon<1$.
Let $X$ and $Y$ be metric spaces.
A map $f:X\to Y$ is called an \textit{$\varepsilon$-almost isometry} if it is surjective and $||f(x)f(y)|/|xy|-1|<\varepsilon$ for any $x,y\in X$.
Let $U$ be an open subset of $X$.
A map $f:U\to Y$ is called an \textit{$\varepsilon$-open map} if for any $x\in U$ and $v\in Y$ such that $\bar B(x,\varepsilon^{-1}|f(x)v|)\subset U$, there exists $y\in U$ such that $f(y)=v$ and $\varepsilon|xy|\le|f(x)v|$.

For a $(k,\delta)$-strainer $\{(a_i,b_i)\}_{i=1}^k$ in $M$, we call $f=(|a_1\cdot|,\dots,|a_k\cdot|):M\to\mathbb R^k$ the \textit{distance coordinate} associated with this strainer.
The above two lemmas imply the following:

\begin{prop}\label{prop:str}
Let $f:M\to\mathbb R^k$ be the distance coordinate associated with a $(k,\delta)$-strainer at $p$ with length $>\ell$.
Then,
\begin{enumerate}
\item if $k=n$, then $f$ is a $\varkappa(\delta)$-almost isometry from $B(p,\ell\delta)$ to an open subset of $\mathbb R^n$;
\item if $k<n$, then $f$ is a $(1+\varkappa(\delta))$-Lipschitz and $(1-\varkappa(\delta))$-open map on $B(p,\ell\delta)$.
\end{enumerate}
\end{prop}

For the proof of (1), see \cite[9.4]{BGP} or \cite[10.9.16]{BBI}.
For the proof of (2), see \cite[3.3]{F} for instance.

\subsection{Lemmas from Perelman's paper}\label{sec:lem}

We list a few lemmas on spaces of curvature $\ge1$ from \cite[\S2]{P1}, which will be used in \S\ref{sec:per} and \S\ref{sec:up}.

Let $\Sigma$ be an $(n-1)$-dimensional Alexandrov space with curvature $\ge1$.
Here we always assume that $\delta$ is sufficiently small compared with $\varepsilon$ (see \S\ref{sec:note}).

\begin{lem}[{\cite[2.2]{P1}}]\label{lem:cnt}
$\Sigma$ cannot contain $n+2$ compact subsets $\{A_i\}_{i=1}^{n+2}$ such that $|A_iA_j|>\pi/2-\delta$ for any $i\neq j$ and $|A_1A_i|>\pi/2+\varepsilon$ for any $i\ge 3$.
\end{lem}

\begin{lem}[{\cite[2.3, 2.4]{P1}}]\label{lem:dir}
\begin{enumerate}
\item Let $\{A_i\}_{i=1}^{k+2}$ be $k+2$ compact subsets ($0\le k\le n-1$) of $\Sigma$ such that $|A_iA_j|>\pi/2-\delta$ for any $i\neq j$ and $|A_1A_i|>\pi/2+\varepsilon$ for any $i\ge 2$.
Then, there is a point $\xi\in \Sigma$ such that $|\xi A_i|=\pi/2$ for any $i\ge 3$ and $|\xi A_1|>\pi/2+c(\varepsilon)$, $|\xi A_2|<\pi/2-c(\varepsilon)$.
\item (1) holds true if we replace the assumption $|A_1A_2|>\pi/2+\varepsilon$ by $|A_1A_2|>\pi/2-\delta$ and the conclusion $|\xi A_1|>\pi/2+c(\varepsilon)$ by $|\xi A_1|>\pi/2-\varkappa(\delta)$.
\item Under the assumptions of (1), there is a point $\xi\in\Sigma$ such that $|\xi A_i|=\pi/2$ for any $i\ge3$ and $|\xi A_2|>\pi/2+c(\varepsilon)$.
\end{enumerate}
\end{lem}

\begin{lem}[{\cite[2.5.2]{P1}}]\label{lem:vol}
Let $A$ be a subset of $\Sigma$.
Let $A(\pi/2\pm a)$ denote the set of points $x\in\Sigma$ such that $||Ax|-\pi/2|\le a$.
Then, the number of $\omega$-discrete points contained in $A(\pi/2\pm a)$ is at most $Ca/\omega^{n-1}$ for any $0<\omega\le a$.
\end{lem}

\section{Construction of a global map}\label{sec:con}

In this section, we construct the global map $f$ of Theorem \ref{thm:subm} and prove the inequality \eqref{eq:main}.

\begin{thm}[cf.\ {\cite[0.2]{Y2}}]\label{thm:con}
Let $X$ be a $k$-dimensional Alexandrov space such that every point has a $(k,\delta)$-strainer with length $>\ell$.
Let $M$ be an $n$-dimensional Alexandrov space and $g:M\to X$ a $\mu$-approximation.
Suppose $\mu<c(\ell\delta^2)\ll\ell\delta^2$.
Then, there exists a map $f:M\to X$, which is $C\mu$-close to $g$, satisfying the following property:

Let $\{(a_i,b_i)\}_{i=1}^k$ be a $(k,\delta)$-strainer at $p\in X$ such that $|a_ip|=\ell\delta$ for all $i$.
Let $\varphi:X\to\mathbb R^k$ be its distance coordinate and $\hat\varphi:M\to\mathbb R^k$ a lift of $\varphi$.
Then, we have
\begin{equation}\label{eq:con}
\bigl|(\varphi\circ f(x)-\varphi\circ f(y))-(\hat\varphi(x)-\hat\varphi(y))\bigr|<\varkappa(\delta)|xy|
\end{equation}
for any $x,y\in B(\hat p,\ell\delta^2)$, where $\hat p\in M$ denotes a lift of $p$ (i.e.\ $|pg(\hat p)|<\mu$).
\end{thm}

\begin{rem}\label{rem:con1}
The inequality \eqref{eq:con} immediately implies that $f$ is a $\varkappa(\delta)$-almost isometry when $k=n$ (cf.\ \cite[9.8]{BGP}, \cite[3.1]{S}, \cite{WSS}) and is a $(1-\varkappa(\delta))$-open, $\varkappa(\delta)$-almost Lipschitz submersion when $k<n$.
See Propositions \ref{prop:reg} and \ref{prop:int}.
\end{rem}

\begin{rem}\label{rem:con2}
The inequality \eqref{eq:con} actually holds for the almost Lipschitz submersion constructed by Yamaguchi \cite{Y2} (use \cite[4.6, 4.13]{Y2}).
His construction is based on an embedding of $X$ to the Hilbert space $L^2(X)$ of all $L^2$-functions on $X$ and the existence of a ``tubular neighborhood'' of the image of $X$.
Here we give a more direct proof.
Our construction is a generalization of that of the almost isometry in \cite[3.1]{S} when $k=n$.
However, both constructions are based on the same idea of gluing local distance coordinates.
\end{rem}

\begin{proof}[Proof of Theorem \ref{thm:con}]
We denote by $\hat a\in M$ a lift of $a\in X$ with respect to $g$.
Set $r:=\ell\delta^2$.
Let $\{p_j\}_{j=1,2,\dots}$ be a maximal $r/2$-discrete net of $X$ (possibly infinite).
Take a $(k,\delta)$-strainer $\{(\alpha_i^j,\beta_i^j)\}_{i=1}^k$ with length $>\ell$ for each $p_j$ and let $a_i^j$, $b_i^j$ be points on shortest paths $p_j\alpha_i^j$, $p_j\beta_i^j$ at distance $\ell\delta$ from $p_j$, respectively. 
Set
\begin{align*}
&\varphi_j:=(|a_1^j\cdot|,\dots,|a_k^j\cdot|),&&U_j:=B(p_j,r),\quad\lambda U_j:=B(p_j,\lambda r),\\
&\hat\varphi_j:=(|\hat a_1^j\cdot|,\dots,|\hat a_k^j\cdot|),&&\hat U_j:=B(\hat p_j,r),\quad\lambda\hat U_j:=B(\hat p_j,\lambda r)
\end{align*}
for $\lambda>0$.
Since $\mu\ll r$, $\{\hat U_j\}_j$ is a covering of $M$ and $\{(\hat a_i^j,\hat b_i^j)\}_{i=1}^k$ is a $(k,\varkappa(\delta))$-strainer for $\hat p_j$.
By Proposition \ref{prop:str}, $\varphi_j$ is a $\varkappa(\delta)$-almost isometry from $10U_j$ to an open subset of $\mathbb R^k$ and $\hat\varphi_j$ is a $(1+\varkappa(\delta))$-Lipschitz map on $3\hat U_j$.
Hence, we can define $\varphi_j^{-1}\circ\hat\varphi_j$ on $3\hat U_j$.
Note that $|\varphi_j^{-1}\circ\hat\varphi_j,g|<C\mu$.
We take an average of them to obtain a global map.
Define $f_j:\bigcup_{j'=1}^j\hat U_{j'}\to X$ inductively as follows:
\[f_1:=\varphi_1^{-1}\circ\hat\varphi_1:\hat U_1\to X\]
and for $j\ge2$,
\begin{equation*}
f_j:=
\begin{cases}
\hfil f_{j-1} & \text{on}\hfil\bigcup_{j'=1}^{j-1}\hat U_{j'}\setminus2\hat U_j\\
\varphi_j^{-1}\circ((1-\chi_j)\varphi_j\circ f_{j-1}+\chi_j\hat\varphi_j) & \text{on}\quad\bigcup_{j'=1}^{j-1}\hat U_{j'}\cap(2\hat U_j\setminus\hat U_j)\\
\hfil\varphi_j^{-1}\circ\hat\varphi_j & \text{on}\hfil\hat U_j.
\end{cases}
\end{equation*}
Here $\chi_j:=\chi(|\hat p_j\cdot|/r)$, where $\chi:[0,\infty)\to[0,1]$ is a smooth function such that $\chi\equiv1$ on $[0,1]$ and $\chi\equiv0$ on $[2,\infty]$.
Note that $\chi_j$ is $L/r$-Lipschitz for some constant $L$.

Then, it easily follows that $|f_j,g|<C\mu$ by the induction on $j$ (in particular, the above definition of $f_j$ works).
Note that the number of the induction steps at each point in the domain of $f_j$ is uniformly bounded above.
Indeed, since $\{p_j\}_j$ is $r/2$-discrete and $\mu\ll r$, the multiplicity of the covering $\{2\hat U_j\}_j$ is bounded above by some constant depending only on $k$.
We define $f(x):=f_{N_x}(x)$ for $x\in M$, where $N_x:=\max\{j\mid x\in2\hat U_j\}$.

Now, we show the inequality \eqref{eq:con}.
Let $p$, $\hat p$, $\varphi$, $\hat\varphi$ be as in the assumption.
Set $\lambda\hat U:=B(\hat p,\lambda r)$ as above.
We prove by the induction on $j$ that
\begin{equation}\label{eq:ind}
\bigl|(\varphi\circ f_j(x)-\varphi\circ f_j(y))-(\hat\varphi(x)-\hat\varphi(y))\bigr|<\varkappa(\delta)|xy|
\end{equation}
for $x,y\in 3\hat U\cap\dom(f_j)$, where $\dom(f_j)=\bigcup_{j'=1}^j\hat U_{j'}$.
Note that we may assume $|xy|<r$ since $|f_j,g|<C\mu$ and $\mu\ll r$.

First, we prove the inequality \eqref{eq:ind} for the special case $\varphi=\varphi_j$ and $p=p_j$.
The base case $j=1$ is trivial.
Suppose $j\ge2$.
Let us consider the case $x,y \in(2\hat U_j\setminus\hat U_j)\cap\dom(f_j)$ (the other cases are similar).
Then, we have
\begin{align*}
&(\varphi_j\circ f_j(x)-\varphi_j\circ f_j(y))-(\hat\varphi_j(x)-\hat\varphi_j(y))\\
&=(1-\chi_j(x))(\varphi_j\circ f_{j-1}(x)-\hat\varphi_j(x))-(1-\chi_j(y))(\varphi_j\circ f_{j-1}(y)-\hat\varphi_j(y))\\
&=(1-\chi_j(x))((\varphi_j\circ f_{j-1}(x)-\varphi_j\circ f_{j-1}(y))-(\hat\varphi_j(x)-\hat\varphi_j(y)))\\
&\qquad\qquad\qquad\qquad\qquad\qquad\qquad-(\chi_j(x)-\chi_j(y))(\varphi_j\circ f_{j-1}(y)-\hat\varphi_j(y)).
\end{align*}
The norm of the first term of the last formula is less than $\varkappa(\delta)|xy|$ by the induction hypothesis.
The same is true for the second term since $\chi_j$ is $L/r$-Lipschitz and $|\varphi_j\circ f_{j-1},\hat\varphi_j|<C\mu$, where $\mu\ll r$.

Next, we consider the general case.
Lemma \ref{lem:str} implies that the inequality \eqref{eq:ind} is equivalent to
\begin{equation}\label{eq:ind'}
\left|\cos\angle a_if_j(x)f_j(y)\cdot\frac{|f_j(x)f_j(y)|}{|xy|}-\cos\angle\hat a_ixy\right|<\varkappa(\delta)
\end{equation}
for all $1\le i\le k$.
By the induction hypothesis, we may assume that $x,y\in3\hat U\cap \dom(f_j)\cap3\hat U_j$ (recall $|xy|<r$).
Fix $1\le i \le k$.
First, by the strainer $\{(a_{i'}^j,b_{i'}^j)\}_{i'=1}^k$ at $f_j(x)$ and Lemma \ref{lem:susp}, $\Sigma_{f_j(x)}$ is $\varkappa(\delta)$-close to $\mathbb S^{k-1}$ in the Gromov-Hausdorff distance.
Thus, we have
\[\left|\cos\angle a_if_j(x)f_j(y)-\sum_{i'=1}^k\cos\angle a_if_j(x)a_{i'}^j\cdot\cos\angle a_{i'}^jf_j(x)f_j(y)\right|<\varkappa(\delta).\]
Next, since $|f_j,g|<C\mu$ and $\mu\ll r$, we have $|\tilde\angle a_if_j(x)a_{i'}^j-\tilde\angle\hat a_ix\hat a_{i'}^j|<\varkappa(\delta)$ for all $1\le i'\le k$.
Recall that $a_{i'}^j$, $b_{i'}^j$ are on the shortest paths $p_j\alpha_{i'}^j$, $p_j\beta_{i'}^j$ at distance $\ell\delta$ from $p_j$, where $\{(\alpha_{i'}^j,\beta_{i'}^j)\}_{i'=1}^k$ is a $(k,\delta)$-strainer with length $>\ell$ at $p_j$.
Hence, by using Lemma \ref{lem:str} twice, we have
\[\left|\angle a_if_j(x)a_{i'}^j-\angle\hat a_ix\hat a_{i'}^j\right|<\varkappa(\delta)\]
for all $1\le i'\le k$.
Finally, by the strainer $\{(\hat a_{i'}^j,\hat b_{i'}^j)\}_{i'=1}^k$ at $x$ and Lemma \ref{lem:susp} again, $\Sigma_x$ is $\varkappa(\delta)$-close to a $k$-fold spherical suspension $S^k(\Sigma)$.
Furthermore, this approximation sends $(\hat a_i)'_x$ into the $\varkappa(\delta)$-neighborhood of $\mathbb S^{k-1}$ in $S^k(\Sigma)$ since the previous inequality implies that $\sum_{i'=1}^k\cos^2\angle\hat a_ix\hat a_{i'}^j\approx\sum_{i'=1}^k\cos^2\angle a_if_j(x)a_{i'}^j\approx1$ (roughly speaking, $(\hat a_i)'_x$ is a ``horizontal direction'').
Therefore, we have
\[\left|\cos\angle\hat a_ixy-\sum_{i'=1}^k\cos\angle\hat a_ix\hat a_{i'}^j\cdot\cos\angle\hat a_{i'}^jxy\right|<\varkappa(\delta).\]

Combining the above three inequalities with the inequality \eqref{eq:ind'} for the special case $\varphi=\varphi_j$ and $p=p_j$, we obtain the general one.
\end{proof}

\section{Modification of Perelman's fibration theorem}\label{sec:per}

In this section, we prove Theorem \ref{thm:main} and Corollary \ref{cor:main}.
We generalize the notion of noncritical maps introduced by Perelman \cite{P1} in terms of the inequality \eqref{eq:con} and prove the fibration theorem \cite[1.4.1]{P1} for such generalized noncritical maps.
Although the proof is almost the same as the original one, we give the details because of the difficulty of the original proof.
Hence, most of the contents of this section are just slight modifications of those of \cite[\S3]{P1}.

Here we use another positive number $\varepsilon$ in addition to $\delta$.
Note that $\varkappa$ in this section may depend on $\varepsilon$.
Furthermore, we assume that $\delta$ is much smaller than $\varepsilon$ and every $c(\varepsilon)$ in this section and that so is $\varkappa(\delta)$.
We use the maximum norm on $\mathbb R^k$ unless otherwise stated.
See \S\ref{sec:note} for the notation and conventions.

The following definition is the key to generalizing the fibration theorem.
For a point $p$ and subsets $A$, $B$ in an Alexandrov space, $\tilde\angle ApB$ denotes the comparison angle at $p$ of the comparison triangle with sidelengths $|pA|$, $|pB|$ and $|AB|$ if it exists; otherwise $\tilde\angle:=0$.

\begin{dfn}[cf.\ {\cite[3.1]{P1}}]\label{dfn:nc}
Let $U$ be an open subset of an Alexandrov space $M$.
A map $f=(f_1,\dots,f_k):U\to\mathbb R^k$ is said to be \textit{$(\varepsilon,\delta)$-noncritical (in the generalized sense)} at $p\in U$ if it satisfies the following conditions:
\begin{enumerate}
\item Each $f_i$ satisfies the following inductive condition:
there exists a function $g_i:U\to\mathbb R$ such that
\begin{equation}\label{eq:nc}
\bigl|(f_i(x)-f_i(y))-(g_i(x)-g_i(y))\bigr|<\delta|xy|
\end{equation}
for any $x,y\in U$, and
\[g_i=\inf_\gamma g_{i\gamma},\quad g_{i\gamma}=\varphi_{i\gamma}(|A_{i\gamma}\cdot|)+\sum_{l=1}^{i-1}\varphi_{i\gamma}^l(f_l(\cdot))+c_{i\gamma},\]
where $c_{i\gamma}\in\mathbb R$, $A_{i\gamma}$ are compact subsets of $M$, $\varphi_{i\gamma}$ and $\varphi_{i\gamma}^l$ have right and left derivatives, $\varphi_{i\gamma}^l$ are $\varepsilon^{-1}$-Lipschitz functions, and $\varphi_{i\gamma}$ are increasing functions such that $\varphi_{i\gamma}(0)=0$ and $\varepsilon|x-y|\le|\varphi_{i\gamma}(x)-\varphi_{i\gamma}(y)|\le\varepsilon^{-1}|x-y|$.
\item The sets of indices $\Gamma_i(p):=\{\gamma\mid g_i(p)=g_{i\gamma}(p)\}$ satisfy $\#\Gamma_i(p)<\varepsilon^{-1}$.
Furthermore, there exists $\rho=\rho(p)>0$ such that for all $i$
\[g_i(x)<g_{i\gamma}(x)-\rho\]
for $x\in B(p,\rho)$ and $\gamma\notin\Gamma_i(p)$ (we assume $B(p,\rho)\subset U$).
\item $\tilde\angle A_{i\alpha}pA_{j\beta}>\pi/2-\delta$ for all $i\neq j$, $\alpha\in\Gamma_i(p)$, $\beta\in\Gamma_j(p)$.
\item There exists a point $w=w(p)\in M$ such that $\tilde\angle A_{i\gamma}pw>\pi/2+\varepsilon$ for all $i$ and $\gamma\in\Gamma_i(p)$.
\end{enumerate}
\end{dfn}

We have added the inequality \eqref{eq:nc} to the original definition.
In particular, each $f_i$ is not necessarily defined by distance functions unlike $g_i$ and does not always have directional derivatives.
However, the inequality \eqref{eq:nc} guarantee that their difference quotients are almost equal.
In case $f_i\equiv g_i$, our definition coincides with the original one.
Note that each $g_i$ is defined by $f_l$ ($1\le l\le i-1$) but not by $g_l$ (this is necessary to prove Proposition \ref{prop:per1} below).

The purpose of this section is to prove the following generalized fibration theorem:

\begin{thm}[cf.\ {\cite[1.4.1]{P1}}]\label{thm:per}
Let $U$ be a domain of an $n$-dimensional Alexandrov space $M$ such that $\vol_{n-1}\Sigma_p>\varepsilon$ for any $p\in U$.
If a map $f:U\to\mathbb R^k$ is proper and $(\varepsilon,\delta)$-noncritical in the generalized sense for some $\delta\ll\varepsilon$, then it is a locally trivial fibration.
\end{thm}

\begin{rem}\label{rem:per}
The term ``noncritical'' (not ``$(\varepsilon,\delta)$-noncritical'') in the original statement \cite[1.4.1]{P1} includes the assumption on the volume of spaces of directions (see \cite[3.7]{P1} or Definition \ref{dfn:ncv} below).
However, we do not use it here to emphasize the dependence of the fibration theorem on the volume of spaces of directions.
\end{rem}

Let us first recall the structure of the proof of the original fibration theorem.
It consists of two parts: the first half is its geometric part, which is presented in \cite[\S3]{P1}, and the second half is its topological part, which is presented in \cite[\S1]{P1}.
Furthermore, as remarked in \cite[1.3]{P1}, all the arguments in the topological part are based only on the properties of noncritical maps established in the geometric part.
In particular, we do not have to go back to the definition of noncritical maps in the topological part.
Thus, in order to prove the generalized fibration theorem, it is sufficient to verify that all the propositions about noncritical maps in the geometric part hold true for the generalized ones.
Indeed, the same proofs work well by using the inequality \eqref{eq:nc}.
Nevertheless, we give the details to explain how and where our modification works because the original proofs are very complicated.

From here (until Definition \ref{dfn:ncv}), we proceed in parallel with \cite[\S3]{P1}.
We first prove the openness of a generalized noncritical map (note that the set of generalized noncritical points of a given map is clearly open).
The following proof is a good demonstration of the usefulness of the inequality \eqref{eq:nc} and is actually the only place where we use it in the proof of the generalized fibration theorem (see also Remark \ref{rem:open} below).

\begin{prop}[cf.\ {\cite[3.2]{P1}}]\label{prop:open}
Let $U$ be an open subset of an $n$-dimensional Alexandrov space $M$.
Let $f:U\to\mathbb R^k$ be $(\varepsilon,\delta)$-noncritical in the generalized sense at any point of $U$.
Then, $k\le n$ and $f$ is $c(\varepsilon)$-open (and is locally $c(\varepsilon)^{-1}$-Lipschitz).
Furthermore, if $k=n$, then $f$ is a local bi-Lipschitz homeomorphism.
\end{prop}

We need the following elementary lemma to prove the $c(\varepsilon)$-openness, which is a metric version of \cite[2.1.1]{P1}.
The proof is easy and will be omitted.

\begin{lem}[cf.\ {\cite[2.1.1]{P1}}]\label{lem:open}
Let $f:U\to\mathbb R^k$ be a continuous (not necessarily differentiable) map defined on an open subset of an Alexandrov space.
Let $\|\cdot\|$ be an arbitrary norm on $\mathbb R^k$.
Suppose that for any $x\in U$ and $v\in\mathbb R^k$ with $f(x)\neq v$, there exists a point $y\in U$ with $y\neq x$ such that
\[\|f(y)-v\|-\|f(x)-v\|\le-\varepsilon|xy|.\]
Then, $f$ is $\varepsilon$-open with respect to the norm $\|\cdot\|$.
\end{lem}

\begin{proof}[Proof of Proposition \ref{prop:open}]
The inequality $k\le n$ immediately follows from Lemma \ref{lem:cnt} and Definition \ref{dfn:nc}(3),(4).
Let us prove the $c(\varepsilon)$-openness of $f$.
We show that for any $p\in U$ and $1\le i\le k$, there exist $q_i^\pm\in U$ arbitrarily close to $p$ such that
\begin{equation}\label{eq:open}
\begin{aligned}
&|f_j(q_i^\pm)-f_j(p)|<\varkappa(\delta)|pq_i^\pm|&&\text{for $j<i$},\\
&c(\varepsilon)|pq_i^\pm|<\pm(f_j(q_i^\pm)-f_j(p))<c(\varepsilon)^{-1}|pq_i^\pm|&&\text{for $j=i$},\\
&|f_j(q_i^\pm)-f_j(p)|<c(\varepsilon)^{-1}|pq_i^\pm|&&\text{for $j>i$}.
\end{aligned}
\end{equation}
Let us first explain how to finish the proof, assuming the above.
Define a norm $\|v\|:=\sum_{j=1}^kc_j(\varepsilon)|v_j|$ for $v=(v_1\dots,v_k)\in\mathbb R^k$, where $c_j(\varepsilon)\ll c_{j-1}(\varepsilon)$ (the choices of $c_j(\varepsilon)$ depend on the constants $c(\varepsilon)$ in the inequalities \eqref{eq:open} so that the following argument holds).
By Lemma \ref{lem:open}, it suffices to show that for any $p\in U$ and $v\in\mathbb R^k$ with $f(p)\neq v$, there is $q\in U$ with $q\neq p$ such that
\[\|f(q)-v\|-\|f(p)-v\|\le-c(\varepsilon)|pq|.\]
Without loss of generality, we may assume $f_i(p)<v_i$ for some $i$.
Choose $q_i^+$ so close to $p$ that $f_i(q_i^+)<v_i$ and put $q:=q_i^+$.
Then, the inequalities \eqref{eq:open} and the suitable choices of $c_j(\varepsilon)$ yield the desired inequality (note $\varkappa(\delta)\ll c_j(\varepsilon)$).

Now, we prove the inequalities \eqref{eq:open}.
Set $A_j'(p):=\bigcup_{\gamma\in\Gamma_j(p)}(A_{j\gamma})'_p$ for each $1\le j\le k$.
Fix $1\le i\le k$.
Then, by Definition \ref{dfn:nc}(3),(4) and Lemma \ref{lem:dir}(1),(3), we get $\xi_i^\pm\in\Sigma_p$ such that
\begin{align*}
\pm(|A_i'(p)\xi_i^\pm|-\pi/2)>c(\varepsilon),\quad|A_j'(p)\xi_i^\pm|=\pi/2
\end{align*}
for any $j\neq i$.
Choose $q_i^\pm\in U$ near $p$ such that $(q_i^\pm)'_p$ is sufficiently close to $\xi_i^\pm$.
We show the inequalities \eqref{eq:open} by the induction on $j$.
Here is an outline.
Fix $j$ and suppose that \eqref{eq:open} hold for $j'<j$.
Then, by the definition of $g_j$ (see Definition \ref{dfn:nc}(1)), the inequality \eqref{eq:open} for $j$ with $f_j$ replaced by $g_j$ holds.
Together with the inequality \eqref{eq:nc}, this implies \eqref{eq:open} for $j$.
Let us demonstrate this in the case where $k=3$, $i=2$ and the sign is $+$.
Put $\xi:=\xi_2^+$ and $q:=q_2^+$.
First,
\begin{align*}
g_1(q)-g_1(p)&=\min_{\gamma\in\Gamma_1(p)}\bigl\{\varphi_{1\gamma}(|A_{1\gamma}q|)-\varphi_{1\gamma}(|A_{1\gamma}p|)\bigr\}\\
&=\min_{\gamma\in\Gamma_1(p)}\bigl\{-\varphi_{1\gamma}'(|A_{1\gamma}p|)\cos|(A_{1\gamma})'_p\xi|\cdot|pq|+o(|pq|)\bigr\},
\end{align*}
where $\varphi_{1\gamma}'$ denotes the right or left derivative of $\varphi_{1\gamma}$.
Since $\varphi_{1\gamma}'\ge0$ and $|A_1'(p)\xi|=\pi/2$, we have
\[|g_1(q)-g_1(p)|<\delta|pq|\]
provided $q$ is sufficiently close to $p$.
Then, the inequality \eqref{eq:nc} implies
\[|f_1(q)-f_1(p)|<2\delta|pq|.\]
Second,
\begin{align*}
&g_2(q)-g_2(p)\\
&=\min_{\gamma\in\Gamma_2(p)}\bigl\{\varphi_{2\gamma}(|A_{2\gamma}q|)-\varphi_{2\gamma}(|A_{2\gamma}p|)+\varphi_{2\gamma}^1(f_1(q))-\varphi_{2\gamma}^1(f_1(p))\bigr\}\\
&=\min_{\gamma\in\Gamma_2(p)}\bigl\{-\varphi_{2\gamma}'(|A_{2\gamma}p|)\cos|(A_{2\gamma})'_p\xi|\cdot|pq|+o(|pq|)+\varphi_{2\gamma}^1(f_1(q))-\varphi_{2\gamma}^1(f_1(p))\bigr\}.
\end{align*}
Since $\varphi_{2\gamma}^1$ is $\varepsilon^{-1}$-Lipschitz, the inequality obtained in the previous step yields $|\varphi_{2\gamma}^1(f_1(q))-\varphi_{2\gamma}^1(f_1(p))|<\varkappa(\delta)|pq|$.
Furthermore, since $\varepsilon\le\varphi_{2\gamma}'\le\varepsilon^{-1}$ and $|A_2'(p)\xi|>\pi/2+c(\varepsilon)$, we have
\[c(\varepsilon)|pq|<g_2(q)-g_2(p)<c(\varepsilon)^{-1}|pq|\]
provided $q$ is sufficiently close to $p$ (note $\varkappa(\delta)\ll c(\varepsilon)$).
Then, the inequality \eqref{eq:nc} implies
\[c(\varepsilon)|pq|<f_2(q)-f_2(p)<c(\varepsilon)^{-1}|pq|.\]
Finally,
\begin{align*}
&g_3(q)-g_3(p)\\
&=\min_{\gamma\in\Gamma_3(p)}\Bigl\{\varphi_{3\gamma}(|A_{3\gamma}q|)-\varphi_{3\gamma}(|A_{3\gamma}p|)+\sum_{l=1,2}\bigl(\varphi_{3\gamma}^l(f_l(q))-\varphi_{3\gamma}^l(f_l(p))\bigr)\Bigr\}.
\end{align*}
Since both $\varphi_{3\gamma}$ and $\varphi_{3\gamma}^l$ are $\varepsilon^{-1}$-Lipschitz, we have
\[|g_3(q)-g_3(p)|<c(\varepsilon)^{-1}|pq|.\]
Then, the inequality \eqref{eq:nc} implies
\[|f_3(q)-f_3(p)|<c(\varepsilon)^{-1}|pq|.\]
This completes the proof of the inequalities \eqref{eq:open} in our special case.
The general case is similar.
Note that the local $c(\varepsilon)^{-1}$-Lipschitzness of $f$ also follows from a similar inductive argument.

Next, we consider the case $k=n$.
It suffices to show that $f$ is injective near each $p\in U$.
Let $\rho=\rho(p)>0$ and $w=w(p)\in M$ be as in Definition \ref{dfn:nc} and take $0<r<\min_{i,\gamma\in\Gamma_i(p)}\{\rho,\delta|wp|,\delta|A_{i\gamma}p|\}$.
Suppose there exist two distinct points $x,y\in B(p,r)$ such that $f(x)=f(y)$.
We may assume $|wx|\le|wy|$.
In particular, we have $\tilde\angle wxy>\pi/2-\varkappa(\delta)$.
On the other hand, by the definition of noncriticality, we have $\tilde\angle A_{i\gamma}xw>\pi/2+c(\varepsilon)$ and $\tilde\angle A_{i\gamma}xA_{j\beta}>\pi/2-\varkappa(\delta)$ for all $i\neq j$ and $\gamma\in\Gamma_i(p)$, $\beta\in\Gamma_j(p)$.
Furthermore, we show $\tilde\angle A_{i\gamma}xy>\pi/2-\varkappa(\delta)$ for all $i$ and $\gamma\in\Gamma_i(x)$ (note $\Gamma_i(x)\subset\Gamma_i(p)$).
Then, these inequalities contradict Lemma \ref{lem:cnt} for $\Sigma_x$ since $\varkappa(\delta)\ll c(\varepsilon)$.
Let $\gamma\in \Gamma_i(x)$.
We may assume $|A_{i\gamma}x|>|A_{i\gamma}y|$; otherwise, we have $\tilde\angle A_{i\gamma}xy>\pi/2-\varkappa(\delta)$.
Then, we have
\begin{align*}
g_i(x)-g_i(y)&\ge g_{i\gamma}(x)-g_{i\gamma}(y)\\
&=\varphi_{i\gamma}(|A_{i\gamma}x|)-\varphi_{i\gamma}(|A_{i\gamma}y|)\\
&\ge\varepsilon(|A_{i\gamma}x|-|A_{i\gamma}y|),
\end{align*}
since $f(x)=f(y)$ and $\varphi_{i\gamma}$ is an increasing function with co-Lipschitz constant $\varepsilon$.
On the other hand, we have $|g_i(x)-g_i(y)|<\delta|xy|$ by the inequality \eqref{eq:nc}.
Together with the above inequality, this implies $\tilde\angle A_{i\gamma}xy>\pi/2-\varkappa(\delta)$.
\end{proof}

\begin{rem}\label{rem:open}
The same arguments as in the above proof show the following two properties.
Let $f:U\to \mathbb R^k$ be $(\varepsilon,\delta)$-noncritical in the generalized sense at $p\in U$.
\begin{enumerate}
\item Let $\xi\in\Sigma_p$ be a direction such that $|A_i'(p)\xi|=\pi/2$ for all $i$, where $A_i'(p)=\bigcup_{\gamma\in\Gamma_i(p)}(A_{i\gamma})'_p$.
Then, for any $q\in U$ sufficiently close to $p$ such that $q'_p$ is sufficiently close to $\xi$, we have $|f(p)f(q)|<\varkappa(\delta)|pq|$.
\item Let $0<r<\min_{i,\gamma\in\Gamma_i(p)}\{\rho(p),\delta|A_{i\gamma}p|\}$.
Then, for any $x,y\in B(p,r)$ with $|f(x)f(y)|<\delta|xy|$, we have $\tilde\angle A_{i\gamma}xy>\pi/2-\varkappa(\delta)$ for all $i$ and $\gamma\in\Gamma_i(x)$.
\end{enumerate}
Indeed, (1) follows from the same inductive argument as in the proof of the inequalities \eqref{eq:open} for $j<i$.
(2) follows from the same argument as in the proof of the local injectivity when $k=n$ (note that the assumption $k=n$ is not needed here and that the condition $f(x)=f(y)$ used in the above proof can be weakened to $|f(p)f(q)|<\delta|pq|$).
Hereafter, we will often use these properties as well as Proposition \ref{prop:open}.
Furthermore, as long as we use them, we do not need the inequality \eqref{eq:nc} anymore in the proof of the generalized fibration theorem.
\end{rem}

From now on, we consider the case $k<n$ unless otherwise stated.
The next proposition is actually unnecessary for the proof of the fibration theorem, but is shown here (it is used to prove the stability theorem \cite[4.3]{P1}).

\begin{prop}[cf.\ {\cite[3.3]{P1}}]\label{prop:conn}
Let $f:U\to\mathbb R^k$ be $(\varepsilon,\delta)$-noncritical in the generalized sense at $p\in U$.
Let $\Pi:=f^{-1}(f(p))$, $\rho_0:=\delta\min_{i,\gamma\in\Gamma_i(p)}\{\rho(p),|w(p)p|,|A_{i\gamma}p|\}$ and $q,r\in\Pi\cap B(p,\rho_0)$.
Then, there exists a curve in $\Pi$ connecting $q$ and $r$ of length $\le c(\varepsilon)|qr|$.
\end{prop}

\begin{proof}
We may assume that $|wq|\le|wr|$.
Then, we have $|w'_q,r'_q|>\pi/2-\varkappa(\delta)$.
By Remark \ref{rem:open}(2), we have $|(A_{i\gamma})'_q,r'_q|>\pi/2-\varkappa(\delta)$ for all $i$ and $\gamma\in\Gamma_i(q)$.
Moreover, by the definition of noncriticality, we have $|(A_{i\gamma})'_q,w'_q|>\pi/2+c(\varepsilon)$ and $|(A_{i\gamma})'_q,(A_{j\beta})'_q|>\pi/2-\varkappa(\delta)$ for all $i\neq j$ and $\gamma\in\Gamma_i(q)$, $\beta\in\Gamma_j(q)$.
Set $A_i'(q):=\bigcup_{\gamma\in\Gamma_i(q)}(A_{i\gamma})'_q$.
Then, applying Lemma \ref{lem:dir}(2) to $w'_q$, $r'_q$, $A_1'(q)$, $\dots$, $A_k'(q)$, we get a direction $\xi\in\Sigma_q$ such that
\[|r'_q\xi|<\pi/2-c(\varepsilon),\quad|A_i'(q)\xi|=\pi/2\]
for all $i$.
Choose $q_1$ near $q$ such that $(q_1)'_q$ is close to $\xi$.
Then, the first inequality above implies $|rq_1|\le|rq|-c(\varepsilon)|qq_1|$.
Furthermore, the second inequalities imply $|f(q)f(q_1)|<\varkappa(\delta)|qq_1|$ (see Remark \ref{rem:open}(1)).
By the $c(\varepsilon)$-openness of $f$, we obtain $q_2\in\Pi$ near $q_1$ such that $c(\varepsilon)|q_1q_2|\le|f(q_1)f(q)|$ (cf.\ \cite[2.1.3]{P1}).
Therefore, we have
\begin{align*}
|rq_2|&\le|rq_1|+|q_1q_2|\\
&\le|rq|-c(\varepsilon)|qq_1|+c(\varepsilon)^{-1}\varkappa(\delta)|qq_1|\\
&\le|rq|-c(\varepsilon)|qq_2|.
\end{align*}
Now, the desired curve is obtained by taking a limit of broken geodesics.
\end{proof}

The following setting will be used in all the arguments below.

\begin{set}[cf.\ {\cite[3.4]{P1}}]\label{set:per}
Let $f:U\to\mathbb R^k$ be $(\varepsilon,\delta)$-noncritical in the generalized sense at $p\in U$, where $U$ is an open subset of an $n$-dimensional Alexandrov space $M$ and $k<n$.
Assume $\vol_{n-1}\Sigma_p\ge\varepsilon$.
Let $w=w(p)\in M$ be as in Definition \ref{dfn:nc}.
Then, the Bishop-Gromov inequality implies that $\vol_{n-1}B({w'_p},\varepsilon/2)\ge c(\varepsilon)$.
Let $0<\omega<\delta$.
Choose points $\{w_\alpha\}_{\alpha=1}^N\subset M$ near $p$ such that 
\begin{itemize}
\item $N\ge L/\omega^{n-1}$, where $L=c(\varepsilon)$;
\item $\tilde\angle w_\alpha pw_\beta>\omega$ ($1\le\alpha\neq\beta\le N$);
\item $\tilde\angle w_\alpha p A_{i\gamma}>\pi/2+\varepsilon/2$ ($1\le\alpha\le N$, $1\le i\le k$, $\gamma\in\Gamma_i(p)$).
\end{itemize}
Let $V$ be a small neighborhood of $p$ such that for any $x\in V$
\begin{itemize}
\item $\tilde\angle w_\alpha xw_\beta>\omega$ ($1\le\alpha\neq\beta\le N$);
\item $\tilde\angle w_\alpha x A_{i\gamma}>\pi/2+\varepsilon/2$ ($1\le\alpha\le N$, $1\le i\le k$, $\gamma\in\Gamma_i(p)$);
\item $\tilde\angle A_{i\alpha}xA_{j\beta}>\pi/2-\delta$ ($1\le i\neq j\le k$, $\alpha\in\Gamma_i(p)$, $\beta\in\Gamma_j(p)$);
\item $|xp|<\min\{\rho(p),\delta|pw_\alpha|,\delta|pA_{i\gamma}|\mid{1\le\alpha\le N,1\le i\le k,\gamma\in\Gamma_i(p)}\}$.
\end{itemize}
Define a function $\sigma:V\to\mathbb R$ by
\[\sigma(x):=\frac1N\sum_{\alpha=1}^N|w_\alpha x|.\]
\end{set}

We first prove the following two lemmas under the above setting.

\begin{lem}[cf.\ {\cite[3.4 Assertion 1]{P1}}]\label{lem:per1}
Under Setting \ref{set:per}, let $x,y\in V$ be such that $|f(x)f(y)|<\delta|xy|$.
Then, one of the following holds:
\begin{enumerate}
\item $(f,|x,\cdot|)$ is $(c(\varepsilon),\varkappa(\delta))$-noncritical in the generalized sense at $y$;
\item $\sigma(y)-\sigma(x)>c(\varepsilon)|xy|$.
\end{enumerate}
\end{lem}

\begin{proof}
Observe that the conditions (1)--(3) of Definition \ref{dfn:nc} for $(f,|x\cdot|)$ at $y$ are clearly satisfied (see Remark \ref{rem:open}(2)).
The rest of the proof is exactly the same as the original one.
Let $a=c(\varepsilon)$ and $b=c(\varepsilon)$ be constants such that $a\ll b\ll L$.
Suppose that (2) does not hold for this $a$.
Then, the mean value of $\cos|(w_\alpha)'_y,X'_y|$ is less than $2a$, where $X'_y$ denotes the set of all directions of shortest paths from $y$ to $x$.
On the other hand, by Lemma \ref{lem:vol}, the number of $\alpha$ such that $||(w_\alpha)'_y,X'_y|-\pi/2|\le b$ is less than $Cb/\omega^{n-1}\ll N$.
Hence, there exists $\alpha$ such that $|(w_\alpha)'_y,X'_y|>\pi/2+b$; otherwise, the mean value of $\cos|(w_\alpha)'_y,X'_y|$ is greater than $c(b)>2a$, a contradiction.
Thus, a point $w(y)$ on a shortest path $w_\alpha y$ sufficiently close to $y$ satisfies the condition (4) of Definition \ref{dfn:nc} for $(f,|x\cdot|)$ at $y$.
\end{proof}

\begin{lem}[cf.\ {\cite[3.4 Assertion 2]{P1}}]\label{lem:per2}
Under Setting \ref{set:per}, let $x,y\in V$ be such that $|f(x)f(y)|<\delta|xy|$ and assume that $x$ is a local maximum point of $\sigma|_{f^{-1}(f(x))}$.
Then, there exists $\alpha$ such that $\tilde\angle w_\alpha yx>\pi/2+c(\varepsilon)$.
\end{lem}

\begin{proof}
It suffices to show that $|(W_\alpha)'_x,y'_x|<\pi/2-c(\varepsilon)$ for some $\alpha$ ,where $(W_\alpha)'_x$ denotes the set of all directions of shortest paths from $x$ to $w_\alpha$.
Let $a=c(\varepsilon)$ and $b=c(\varepsilon)$ be constants such that $a\ll b\ll L$.
Suppose $|(W_\alpha)'_x,y'_x|\ge\pi/2-a$ for all $\alpha$.
Let $\mathcal A_1$ (resp.\ $\mathcal A_2$) be the set of indices $\alpha$ such that $|(W_\alpha)'_x,y'_x|\le\pi/2+a$ (resp.\ $>\pi/2+a$).
Set $W'(x):=\bigcup_{\alpha\in\mathcal A_2}(W_\alpha)'_x$, $A_i'(x):=\bigcup_{\gamma\in\Gamma_i(p)}(A_{i\gamma})'_x$.
Note that $|A_i'(x),y'_x|>\pi/2-\varkappa(\delta)$ by Remark \ref{rem:open}(2).
Then, applying Lemma \ref{lem:dir}(1) to $W'(x)$, $y'_x$, $A_1'(x)$, $\dots$, $A_k'(x)$, we get a direction $\xi\in\Sigma_x$ such that
\[|W'(x)\xi|>\pi/2+c(a),\quad|A_i'(x)\xi|=\pi/2\]
for all $i$.
Let $\mathcal A_3$ be the set of indices $\alpha\in\mathcal A_2$ such that $|(W_\alpha)'_x\xi|>\pi/2+b$.
Note that $\#\mathcal A_1\le Ca/\omega^{n-1}$ and $\#\mathcal A_3\ge(L-Ca-Cb)/\omega^{n-1}$ (see Lemma \ref{lem:vol}).
Choose $x_1$ near $x$ such that $(x_1)'_x$ is close to $\xi$.
Then, we have
\begin{align*}
\sigma(x_1)&\ge\sigma(x)+N^{-1}(\#\mathcal A_3\cdot c(b)-\#\mathcal A_1)|xx_1|\\
&\ge\sigma(x)+N^{-1}\omega^{-(n-1)}((L-Ca-Cb)\cdot c(b)-Ca)|xx_1|\\
&\ge\sigma(x)+c(\varepsilon)|xx_1|
\end{align*}
since $a\ll b\ll L$ and $N\le C/\omega^{n-1}$.
Furthermore, by Remark \ref{rem:open}(1), we have $|f(x_1)f(x)|<\varkappa(\delta)|xx_1|$.
Thus, by the $c(\varepsilon)$-openness of $f$, we can find $x_2\in f^{-1}(f(x))$ near $x$ such that $c(\varepsilon)|x_1x_2|\le\varkappa(\delta)|xx_1|$ (cf.\ \cite[2.1.3]{P1}).
Therefore, we have
\[\sigma(x_2)\ge\sigma(x_1)-|x_1x_2|\ge\sigma(x)+c(\varepsilon)|xx_1|-c(\varepsilon)^{-1}\varkappa(\delta)|xx_1|>\sigma(x).\]
This contradicts the local maximality of $\sigma|_{f^{-1}(f(x))}$ at $x$.
\end{proof}

Let $f:U\to\mathbb R^k$ be $(\varepsilon,\delta)$-noncritical in the generalized sense at $p\in U$.
We say that $f$ is \textit{$(\varepsilon,\delta)$-complementable (in the generalized sense)} at $p$ if there exists a function $f_{k+1}$ defined on a neighborhood of $p$ such that $(f,f_{k+1})$ is $(\varepsilon,\delta)$-noncritical in the generalized sense at $p$.

The next proposition is the key to the proof of the fibration theorem.
For $v\in\mathbb R^k$ and $r>0$, we denote by $I^k(v,r)$ the closed $r$-neighborhood of $v$ in $\mathbb R^k$ with respect to the maximum norm (recall that we use the maximum norm in this section).

\begin{prop}[cf.\ {\cite[3.5]{P1}}]\label{prop:per1}
Let $f:U\to\mathbb R^k$ be $(\varepsilon,\delta)$-noncritical in the generalized sense at $p\in U$ and assume $\vol_{n-1}\Sigma_p\ge\varepsilon$.
Suppose that $f$ is not $(c(\varepsilon),\varkappa(\delta))$-complementable at $p$.
Then, for sufficiently small $R>0$, there exists a function $h$ defined on $U_1=f^{-1}(I^k(f(p),\delta^5R))\cap\bar B(p,R)$ such that
\begin{enumerate}
\item $h(U_1)=[0,R]$ and $h(x)=|px|$ if $|px|>R/2$;
\item $f$ is injective on $S=h^{-1}(0)$;
\item $f$ is $(c(\varepsilon),\varkappa(\delta))$-complementable at any point of $U_1\setminus S$;
\item $(f,h)$ is $(c(\varepsilon,\varkappa(\delta))$-noncritical in the generalized sense at any point $x\in U_1$ such that $|f(x)f(S)|<\frac13\delta^5h(x)$.
\end{enumerate}
\end{prop}

\begin{proof}
The proof is completely the same as the original one.
Let $R>0$ be so small that $U_1$ is contained in the neighborhood $V$ of Setting \ref{set:per}.
Let $M=c(\varepsilon)$ be the constant of Lemma \ref{lem:per1}(2) and define a compact set $S$ by
\begin{equation*}
S:=\left\{x\in U_1\;\middle|\;
\begin{gathered}
\sigma(x)-\sigma(y)\ge M|xy|\ \text{for all}\ y\in U_1\\
\text{satisfying}\ |f(x)f(y)|<\delta|xy|
\end{gathered}
\right\}.
\end{equation*}
Then, $f$ is $(c(\varepsilon),\varkappa(\delta))$-complementable at any point of $U_1\setminus S$ by Lemma \ref{lem:per1}. In particular, $p\in S$ by the assumption.
Furthermore, the $c(\varepsilon)$-openness of $f$ implies that $|f(x)f(y)|\ge M_1|xy|$ for any $x,y\in S$, where $M_1=c(\varepsilon)$.
(Indeed, assume $f$ is $M_2$-open, where $M_2=c(\varepsilon)$, and let $M_1\ll MM_2$.
Suppose that there exist $x,y\in S$ such that $|f(x)f(y)|<M_1|xy|$.
Then, we can find $x_1\in f^{-1}(f(x))$ such that $|x_1y|\ll M|xy|$.
Together with the definition of $S$, this implies $\sigma(x)>\sigma(y)$.
Similarly $\sigma(x)<\sigma(y)$, a contradiction.)
In particular, $f$ is injective on $S$ and $S\subset\bar B(p,M_1^{-1}\delta^5 R)$.

Let us define the function $h$.
Choose a sequence $S_j$ of finite subsets of $S$ such that
\begin{itemize}
\item $S_0:=\{p\}$;
\item $S_j\supset S_{j-1}$ and $f(S_j)$ is a maximal $\delta^{j+5}R$-discrete net in $f(S)$.
\end{itemize}
Then, we define $h(x):=\inf\{h_\gamma(x)\mid p_\gamma\in\bigcup_jS_j\}$ by
\[h_\gamma(x):=\varphi_{\delta^{j+1}R}(|p_\gamma x|)+\sum_{l=1}^k\frac{10}{M_1}|f_l(x)-f_l(p_\gamma)|\]
for $p_\gamma\in S_j\setminus S_{j-1}$ ($j\ge1$) and
\[h_\gamma(x):=\min\left\{\varphi_{\delta R}(|px|)+\sum_{l=1}^k\frac{10}{M_1}|f_l(x)-f_l(p)|,\ \frac12\varphi_{\frac R2}(|px|)+\frac R4\right\}\]
for $p_\gamma=p$, where
\begin{equation*}
\varphi_r(a):=
\begin{cases}
\hfil a, & a\le r\\
2a-r, & a\ge r.
\end{cases}
\end{equation*}
Note that the above definition of $h$ is exactly the same as in the original proof and that it has the same form as the function $g_i$ in Definition \ref{dfn:nc}(1).

Then, it easily follows that $h^{-1}(0)=S$, $h(U_1)=[0,R]$ and $h(x)=|px|$ if $|px|>R/2$.
To verify the conclusion (4), we need the following lemma.

\begin{lem}[{\cite[3.5 Assertion 3]{P1}}]\label{lem:per3}
For $x\in U_1\setminus S$, set $\Gamma(x):=\{\gamma\mid h_\gamma(x)=h(x)\}$.
Suppose $|f(x)f(S)|<\frac13\delta^5h(x)$.
Then, there exists $s\in S$ such that
\[|f(x)f(s)|<\delta^2|xs|,\quad |p_\gamma s|<\delta|xs|\]
for all $\gamma\in\Gamma(x)$.
Moreover, we have $\#\Gamma(x)<c^{-1}(\varepsilon)$.
\end{lem}

We omit the proof since it is based only on the definition of $h$ and the $M_1$-co-Lipschitzness of $f$ on $S$, and does not depend on the definition of noncriticality (see the original proof).

Let us show the conclusion (4).
Let $x\in U_1\setminus S$ and $s\in S$ be as above. 
Then, by Lemma \ref{lem:per2}, there exists $\alpha$ such that $\tilde\angle w_\alpha xs>\pi/2+c(\varepsilon)$.
Together with the second inequality in the above lemma, this implies $\tilde\angle w_\alpha xp_\gamma>\pi/2+c(\varepsilon)$ for all $\gamma\in\Gamma(x)$.
Furthermore, the two inequalities in the above lemma together with Remark \ref{rem:open}(2) imply that $\tilde\angle A_{i\beta}xp_\gamma>\pi/2-\varkappa(\delta)$ for all $\beta\in\Gamma_i(x)$ and $\gamma\in\Gamma(x)$.
Thus, all the conditions of Definition \ref{dfn:nc} are satisfied for $(f,h)$ at $x$.
\end{proof}

The last proposition is a refinement of the previous one on each fiber.

\begin{prop}[cf.\ {\cite[3.6]{P1}}]\label{prop:per2}
Let $f:U\to\mathbb R^k$ be $(\varepsilon,\delta)$-noncritical in the generalized sense at $p\in U$ and assume $\vol_{n-1}\Sigma_p\ge\varepsilon$.
Let $R>0$ be so small that $U_1=f^{-1}(I^k(f(p),\delta^5R))\cap\bar B(p,R)$ is contained in the neighborhood $V$ of Setting \ref{set:per}.
Suppose that 
\begin{itemize}
\item for any $x\in U_1$ such that $\delta R\le|px|\le R$, we have $\sigma(p)-\sigma(x)\ge M|px|$,
\end{itemize}
where $M=c(\varepsilon)$ is the constant of Lemma \ref{lem:per1}(2) (in particular, this is weaker than the condition that $f$ is not $(c(\varepsilon),\varkappa(\delta))$-complementable at $p$).
Then, for any $v\in I^k(f(p),\delta^5R)$, there exists a function $h_v:U_1\to[0,R]$ and a point $o_v\in U_1\cap f^{-1}(v)$ such that
\begin{enumerate}
\item for $x\in U_1\cap f^{-1}(v)$, we have $h_v(x)=R\Leftrightarrow|px|=R$ and $h_v(x)=0\Leftrightarrow x=o_v$;
\item $(f,h_v)$ is $(c(\varepsilon),\varkappa(\delta))$-noncritical in the generalized sense on $U_1\cap f^{-1}(v)\setminus\{o_v\}$.
\end{enumerate}
\end{prop}

\begin{proof}
The proof is completely the same as the original one.
Let $o_v$ be a maximum point of $\sigma|_{U_1\cap f^{-1}(v)}$.
Then, by the assumption and the $c(\varepsilon)$-openness of $f$, we have $|po_v|<\delta R$.
(Indeed, suppose the contrary; then we have $\sigma(p)\ge\sigma(o_v)+M|po_v|$.
On the other hand, we can find $q\in f^{-1}(v)$ such that $|pq|\ll |po_v|$.
Thus, we have $\sigma(q)>\sigma(o_v)$, a contradiction.)
Define
\[h_v(x):=\min\left\{\varphi_{\delta R}(|o_vx|),\ \frac12\varphi_{\frac R2}(|px|)+\frac R4\right\},\]
where $\varphi_r(a)=\max\{a,2a-r\}$ as before.
Then, (1) is clear.
Let us show (2).
Let $x\in U_1\cap f^{-1}(v)\setminus\{o_v\}$.
Then, $(f,|o_v\cdot|)$ is $(c(\varepsilon),\varkappa(\delta))$-noncritical at $x$ by Lemma \ref{lem:per2} and Remark \ref{rem:open}(2).
Similarly, $(f,|p\cdot|)$ is $(c(\varepsilon),\varkappa(\delta))$-noncritical at $x$ if $|px|\ge\delta R$ by the assumption and Lemma \ref{lem:per1}.
On the other hand, $\varphi_{\delta R}(|o_vx|)<\frac12\varphi_{\frac R2}(|px|)+\frac R4$ if $|px|<\delta R$ since $|po_v|<\delta R$.
Thus, $(f,h_v)$ is  $(c(\varepsilon),\varkappa(\delta))$-noncritical at $x$.
\end{proof}

We have now shown all the propositions in \cite[\S3]{P1} for generalized noncritical maps.
Finally, we define the term ``noncritical (in the generalized sense)" (not ``$(\varepsilon,\delta)$-noncritical'') as follows:

\begin{dfn}[cf.\ {\cite[3.7]{P1}}]\label{dfn:ncv}
Let $U$ be a domain in an $n$-dimensional Alexandrov space $M$ such that $\varepsilon_0=\inf_{p\in U}\vol_{n-1}\Sigma_p>0$.
A map $f:U\to\mathbb R^k$ ($0\le k\le n+1$) is said to be \textit{noncritical (in the generalized sense)} at $p$ if it is $(\varepsilon,\delta)$-noncritical in the generalized sense at $p$ for some $\varepsilon<\varepsilon_0$ and $\delta<\Delta_{n,k}(\varepsilon)$, where $\Delta_{n,k}(\varepsilon)$ is defined by reverse induction on $k$ in such a way that $(\varepsilon,\delta)$-noncritical maps $f:U\to\mathbb R^k$ with $\varepsilon<\varepsilon_0$ and $\delta<\Delta_{n,k}(\varepsilon)$ satisfy all the propositions above (including lemmas) and the pairs $(c(\varepsilon),\varkappa(\delta))$ appearing in them satisfy $\varkappa(\delta)<\Delta_{n,k+1}(c(\varepsilon))$.
\end{dfn}

Then, all the properties of noncritical maps listed in \cite[1.3]{P1} are also true for this definition.
As pointed out at the beginning of this discussion, the remaining topological part \cite[1.4--1.5]{P1} of the proof of the fibration theorem is based only on those properties.
Thus, the generalized fibration theorem follows in exactly the same way.
This completes the proof of Theorem \ref{thm:per}.

\begin{proof}[Proof of Theorem \ref{thm:main} and Corollary \ref{cor:main}]
Theorem \ref{thm:main} follows from Theorem \ref{thm:con} and Theorem \ref{thm:per}.
Indeed, the map $\varphi\circ f$ in Theorem \ref{thm:con} is $(c,\varkappa(\delta))$-noncritical in the generalized sense on $B(\hat p,\ell\delta^2)$ (we can find $w$ of Definition \ref{dfn:nc}(4) by Proposition \ref{prop:str}).
Note that $(\varphi\circ f)^{-1}(B(\varphi(p),\ell\delta^2/2))$ has a compact closure in $B(\hat p,\ell\delta^2)$, since $f$ is $C\mu$-close to the $\mu$-approximation $g$, where $\mu\ll\ell\delta^2$ (we use the standard Euclidean metric here).
Therefore, by Theorem \ref{thm:per}, it is homeomorphic to the product $f^{-1}(p)\times B(\varphi(p),\ell\delta^2/2)$, where the second component is given by $\varphi\circ f$.
Moreover, Corollary \ref{cor:main} follows from Lemma \ref{lem:susp} and Remark \ref{rem:susp}.
Indeed,  if $k=n-1$, then $\vol_{n-1}\Sigma_p>c$ for any $p\in M$.
It easily follows from Proposition \ref{prop:conn} and the topological part \cite[1.4--1.5]{P1} of the proof of the fibration theorem that the fiber is homeomorphic to a circle or a closed interval.
\end{proof}

\begin{prob}\label{prob:per}
Is it possible to prove the fibration theorem for noncritical maps without the assumption on the volume of spaces of directions?
If possible, then probably the map $f$ of Theorem \ref{thm:con} is a locally trivial fibration.
\end{prob}

\begin{rem}\label{rem:mor}
A similar modification does not work for Perelman's another proof of the fibration theorem in \cite{P2}, which requires no assumptions on the volume of spaces of directions.
This is because the regularity of maps defined in \cite{P2} is stronger than the above noncriticality in that it does not include an error $\delta$ as in Definition \ref{dfn:nc}.
For example, a map $f=(|A_1\cdot|,\dots,|A_k\cdot|):M\to\mathbb R^k$, where $A_i$ are compact subsets of $M$, is called \textit{$\varepsilon$-regular} at $p\in M$ if it satisfies
\begin{enumerate}
\item $\angle((A_i)'_p,(A_j)'_p)>\pi/2+\varepsilon$ for all $i\neq j$;
\item there exists a direction $\xi\in\Sigma_p$ such that $\angle((A_i)'_p,\xi)>\pi/2+\varepsilon$ for all $i$.
\end{enumerate}
In addition, $f$ is simply called \textit{regular} at $p$ if it is $\varepsilon$-regular at $p$ for some $\varepsilon>0$.
Then, the fibration theorem in \cite{P2} states that every proper regular map defined on a domain of an Alexandrov space is a locally trivial fibration.
However, we cannot weaken the condition (1) to $\angle((A_i)'_p,(A_j)'_p)>\pi/2-\delta$ for any $\delta\ll\varepsilon$.
The reason is as follows.
There is a counterpart \cite[1.3]{P2} of Propositions \ref{prop:per1}, \ref{prop:per2} asserting that if $g:M\to\mathbb R^k$ is regular and incomplementable at $p$, then there exists a nonpositive function $g_{k+1}$ defined on a neighborhood of $p$ such that $(g,g_{k+1})$ is regular on the complement of $g_{k+1}^{-1}(0)$.
Nevertheless, unlike Propositions \ref{prop:per1}, \ref{prop:per2}, even if the given map $g$ is $\varepsilon$-regular at $p$, the new map $(g,g_{k+1})$ is not uniformly $c(\varepsilon)$-regular.
Thus, if we modify the condition (1) as above, the given error $\delta$ may be too big compared with the regularity of the new map $(g,g_{k+1})$.
The same problem occurs if we try to generalize the definition of regularity by using inequality \eqref{eq:con} like Definition \ref{dfn:nc}.
\end{rem}

We conclude this section with the following stability theorem for fibrations.
Note that the stability theorem \cite[4.3]{P1} for framed subsets can be generalized like the fibration theorem just by replacing ``noncritical map'' with ``noncritical map in the generalized sense.''
Indeed, the proof of the stability theorem is based only on the properties of noncritical maps established in \cite[\S3]{P1} like the topological part of the proof of the fibration theorem.

\begin{cor}\label{cor:stab}
Let $X$, $M$ be Alexandrov spaces as in Theorem \ref{thm:main} and suppose in addition that they are compact.
Let $\tilde M$ be an $n$-dimensional Alexandrov space sufficiently close to $M$ (in particular, $\tilde M$ also satisfies the assumption of Theorem \ref{thm:main}; see \cite[7.14]{BGP}).
Let $f:M\to X$ and $\tilde f:\tilde M\to X$ be the fibrations constructed in Theorem \ref{thm:con} (where the approximation $\tilde g:\tilde M\to X$ is assumed to be obtained from the approximations $g:M\to X$ and $\Psi:M\to\tilde M$).
Then, there exists a homeomorphism $\Phi:M\to\tilde M$ close to $\Psi$ respecting $f$, that is, $f=\tilde f\circ\Phi$.
\end{cor}

\begin{proof}
We give only an outline.
Let $\varphi:X\to\mathbb R^k$ be the distance coordinate around $p\in X$ as in Theorem \ref{thm:con}.
Then, $\varphi\circ f$ and $\varphi\circ\tilde f$ are noncritical maps in the generalized sense on some neighborhoods of lifts of $p$ in $M$ and $\tilde M$, respectively.
By the generalized fibration theorem, $(\varphi\circ f)^{-1}(I^k(\varphi(p),\rho))$ is homeomorphic to $f^{-1}(p)\times I^k(\varphi(p),\rho)$ for sufficiently small $\rho>0$, where the second component is given by $\varphi\circ f$ and the fiber $f^{-1}(p)$ is an MCS-space in the sense of \cite[1.1]{P1}.
Furthermore, by the generalized stability theorem mentioned above, there exists a homeomorphism $\Phi_p:(\varphi\circ f)^{-1}(I^k(\varphi(p),\rho))\to(\varphi\circ\tilde f)^{-1}(I^k(\varphi(p),\rho))$ close to $\Psi$ respecting $\varphi\circ f$, provided that $\tilde M$ is sufficiently close to $M$.
Take a finite cover of $M$ by such product neighborhoods with respect to $f$.
Then, the desired homeomorphism is constructed by the same gluing argument as in the proof of Complement to Theorem B in \cite[\S1]{P1}.
\end{proof}

\section{Properties of the fibers}\label{sec:fbr}

In this section, we study the properties of the fibers of the map $f$ of Theorem \ref{thm:con}.
Note that the assumption on the volume of spaces of directions as in Theorem \ref{thm:main} is not required here.
We first remark that the diameters of the fibers of $f$ are very small:

\begin{rem}\label{rem:fbr}
Let $f:M\to X$ be the map of Theorem \ref{thm:con} and $p\in X$.
Since $f$ is $C\mu$-close to the $\mu$-approximation $g$, the fiber $f^{-1}(p)$ is contained in the $C\mu$-neighborhood of a lift $\hat p\in M$ of $p$.
Note that $\mu<c(\ell\delta^2)\ll\ell\delta^2$.
\end{rem}

In view of the inequality \eqref{eq:con}, we mainly deal with the following class of maps in this section:

\begin{dfn}\label{dfn:reg}
Let $M$ be an Alexandrov space and $p\in M$.
Let $\{(a_i,b_i)\}_{i=1}^k$ be a $(k,\delta)$-strainer at $p$ with length $>\ell$ and $g:M\to\mathbb R^k$ its distance coordinate.
Suppose a map $f:B(p,\ell)\to\mathbb R^k$ satisfies
\begin{equation}\label{eq:reg}
\bigl|(f(x)-f(y))-(g(x)-g(y))\bigr|<\delta|xy|
\end{equation}
for any $x,y\in B(p,\ell)$.
Then, we call $f$ a \textit{$\delta$-almost regular map} associated with the strainer $\{(a_i,b_i)\}_{i=1}^k$.
\end{dfn}

While the domain of $f$ may seem to be too large, it is useful for simplicity and is sufficient for our applications.

\begin{rem}\label{rem:reg}
The above definition of an almost regular map is different from that in \cite[11.7]{BGP} (the original definition generalizes the distance coordinate of a strainer from a different point of view).
However, as in the case of noncritical maps, all the claims about almost regular maps in \cite{BGP} hold for the above ones.
This is why we use the same term.
\end{rem}

\begin{rem}\label{rem:regcon}
Let $f:M\to X$ be the map of Theorem \ref{thm:con} and let $p$, $\hat p$, $\varphi$, $\hat \varphi$ be as in Theorem \ref{thm:con}.
Then, by the inequality \eqref{eq:con}, $\varphi\circ f$ is a $\varkappa(\delta)$-almost regular map on $B(\hat p,\ell\delta^2)$ associated with a $(k,\varkappa(\delta))$-strainer at $\hat p$ with length $>\ell\delta^2$.
\end{rem}

The following proposition is a generalization of Proposition \ref{prop:str}:

\begin{prop}\label{prop:reg}
Let $M$ be an $n$-dimensional Alexandrov space and $p\in M$.
Let $f:B(p,\ell)\to\mathbb R^k$ be a $\delta$-almost regular map associated with a $(k,\delta)$-strainer at $p$ with length $>\ell$.
Then,
\begin{enumerate}
\item if $k=n$, then $f$ is a $\varkappa(\delta)$-almost isometry from $B(p,\ell\delta)$ to an open subset of $\mathbb R^n$;
\item if $k<n$, then $f$ is a $(1+\varkappa(\delta))$-Lipschitz and $(1-\varkappa(\delta))$-open map on $B(p,\ell\delta)$.
\end{enumerate}
\end{prop}

The proof is an easy application of the inequality \eqref{eq:reg} (the $(1-\varkappa(\delta))$-openness follows from \cite[3.1]{F} since the distance coordinate $g$ satisfies the assumption of \cite[3.1]{F} on $B(p,\ell\delta)$ and so does $f$).

In particular, by Remark \ref{rem:regcon}, the map $f$ of Theorem \ref{thm:con} is a $\varkappa(\delta)$-almost isometry onto $X$ when $k=n$, and is a $(1+\varkappa(\delta))$-Lipschitz and $(1-\varkappa(\delta))$-open map when $k<n$ (note that $f$ is $C\mu$-close to the $\mu$-approximation $g$ globally, where $\mu\ll\ell\delta^2$).

\subsection{Intrinsic metric of the fibers}\label{sec:int}

In this section, we prove Theorem \ref{thm:fbr}(1).
We first show that an almost regular map is an almost Lipschitz submersion near the strained point.
Recall that for a subset $A$ of an Alexandrov space and $p\in A$, the \textit{space of directions} of $A$ at $p$ is defined as the subset of $\Sigma_p$ consisting of all limit points $\lim_{i\to\infty}(p_i)'_p$, where $p_i\in A$ converges to $p$.

\begin{prop}[cf.\ \cite{Y2}]\label{prop:int}
Let $M$ be an $n$-dimensional Alexandrov space, $p\in M$ and $k<n$.
Let $f:B(p,\ell)\to\mathbb R^k$ be a $\delta$-almost regular map associated with a $(k,\delta)$-strainer $\{(a_i,b_i)\}_{i=1}^k$ at $p$ with length $>\ell$.
Then, $f$ is a $\varkappa(\delta)$-almost Lipschitz submersion on $B(p,\ell\delta)$ in the following sense:
\[\left|\frac{|f(x)f(y)|}{|xy|}-\sin\angle(y'_x,V_x)\right|<\varkappa(\delta),\]
for any $x,y\in B(p,\ell\delta)$ and any direction $y'_x$, where $V_x$ denotes the space of directions of $f^{-1}(f(x))$ at $x$ (and is nonempty).
\end{prop}

In particular, by Remark \ref{rem:regcon}, the map $f$ of Theorem \ref{thm:con} is locally (indeed, globally) a $\varkappa(\delta)$-almost Lipschitz submersion when $k<n$.

\begin{rem}\label{rem:int}
The above definition of an almost Lipschitz submersion is slightly stronger than that in \cite{Y2} (see Theorem \ref{thm:subm}).
\end{rem}

\begin{proof}[Proof of Proposition \ref{prop:int}]
Let $x,y\in B(p,\ell\delta)$.
Then, by Lemma \ref{lem:str} and the inequality \eqref{eq:reg}, we have
\[\left|\frac{|f(x)f(y)|^2}{|xy|^2}-\sum_{i=1}^k\cos^2\angle a_ixy\right|<\varkappa(\delta).\]
By Lemma \ref{lem:susp}, there exists a $\varkappa(\delta)$-approximation from $\Sigma_x$ to a $k$-fold spherical suspension $S^k(\Sigma)$, where $\Sigma$ has curvature $\ge1$ and is nonempty (see Remark \ref{rem:susp}).
Let $\tilde\Sigma\subset\Sigma_x$ be a subset corresponding to $\Sigma\subset S^k(\Sigma)$ via this approximation.
Then, we have
\[\left|\sum_{i=1}^k\cos^2\angle a_ixy-\sin^2\angle(y'_x,\tilde\Sigma)\right|<\varkappa(\delta).\]
Thus, it suffices to show that the Hausdorff distance between $\tilde\Sigma$ and $V_x$ is less than $\varkappa(\delta)$.
The above inequalities immediately imply that $V_x$ is contained in the $\varkappa(\delta)$-neighborhood of $\tilde\Sigma$.
On the other hand, let $\xi\in\tilde\Sigma$.
Choose a point $z$ near $x$ such that $z'_x$ is sufficiently close to $\xi$.
Then, the above inequalities imply that $|f(x)f(z)|<\varkappa(\delta)|xz|$.
By the $(1-\varkappa(\delta))$-openness of $f$, we can find $w\in f^{-1}(f(x))$ such that $(1-\varkappa(\delta))|wz|\le|f(x)f(z)|$.
In particular, we have $|wz|<\varkappa(\delta)|xz|$.
Since $z$ can be chosen arbitrarily close to $x$, we obtain $\angle zxw<\varkappa(\delta)$.
This completes the proof.
\end{proof}

Now, we prove Theorem \ref{thm:fbr}(1).
By Remarks \ref{rem:fbr} and \ref{rem:regcon}, it suffices to show the following:

\begin{cor}[cf.\ {\cite[11.11]{BGP}}]\label{cor:int}
Let $M$ be an $n$-dimensional Alexandrov space, $p\in M$ and $k<n$.
Let $f:B(p,\ell)\to\mathbb R^k$ be a $\delta$-almost regular map associated with a $(k,\delta)$-strainer at $p$ with length $>\ell$.
Then, for any $x,y\in f^{-1}(f(p))\cap B(p,\ell\delta)$, there exists a curve in $f^{-1}(f(p))$ connecting $x$ and $y$ of length $<(1+\varkappa(\delta))|xy|$.
\end{cor}

\begin{proof}
By Proposition \ref{prop:int}, for any $x,y\in f^{-1}(f(p))\cap B(p,\ell\delta)$, there exists $z\in f^{-1}(f(p))$ arbitrarily close to $x$ such that $\angle yxz<\varkappa(\delta)$.
In particular, the first variation formula implies that $|yz|<|yx|-(1-\varkappa(\delta))|xz|$.
Thus, the desired curve is obtained by taking a limit of broken geodesics.
\end{proof}

\subsection{Lower bound for the volume of the fibers}\label{sec:low}

In this section, we prove the left inequality of Theorem \ref{thm:fbr}(2).
By Remark \ref{rem:regcon}, it suffices to show the following:

\begin{prop}\label{prop:low}
Let $M$ be an $n$-dimensional Alexandrov space, $p\in M$ and $k<n$.
Let $f:B(p,\ell)\to\mathbb R^k$ be a $\delta$-almost regular map associated with a $(k,\delta)$-strainer $\{(a_i,b_i)\}_{i=1}^k$ at $p$ with length $>\ell$.
Suppose $\vol_n(B(p,D))>v$.
Then, we have
\[\vol_{n-k}f^{-1}(f(p))>c(D,v,\ell).\]
\end{prop}

\begin{rem}\label{rem:low}
Since $f$ is $(1-\varkappa(\delta))$-open on $B(p,\ell\delta)$, we have $\vol_{n-k}f^{-1}(u)>c(D,v,\ell)$ for any $u\in B(f(p),\ell\delta/2)$
\end{rem}

From now, we fix sufficiently small $\delta$ depending only on $n$ and $\kappa$ (we determine it later).
We first give a lower bound for the diameter of the fiber:

\begin{lem}\label{lem:low}
Under the same assumption as Proposition \ref{prop:low}, there exists a point $q\in f^{-1}(f(p))$ such that $\rho<|pq|<\ell\delta$, where $\rho=c(D,v,\ell)$.
\end{lem}

\begin{proof}
We argue by contradiction.
Suppose that there exists a sequence of $n$-dimensional Alexandrov spaces $(M_j,p_j)$ with $\vol_n B(p_j,D)>v$ and $\delta$-almost regular maps $f_j:B(p_j,\ell)\to\mathbb R^k$ associated with $(k,\delta)$-strainers $\{(a_i^j,b_i^j)\}_{i=1}^k$ at $p_j$ with lengths $>\ell$ such that $\diam f_j^{-1}(f_j(p_j))\cap B(p_j,\ell\delta)\to 0$.
For simplicity, we assume that $|p_ja_i^j|$, $|p_jb_i^j|$ are uniformly bounded above.
Since $\vol_n B(p_j,D)>v$, we may assume that $(M_j,p_j)$ converges to an Alexandrov space $(M,p)$ of dimension $n$ (note that it is different from $(M,p)$ in the statement of Lemma \ref{lem:low}; so is $q$ below).
Furthermore, since lengths $>\ell$, we may assume that $\{(a_i^j,b_i^j)\}_{i=1}^k$ converges to a $(k,2\delta)$-strainer $\{(a_i,b_i)\}_{i=1}^k$ at $p$.
Then, by Lemma \ref{lem:susp}, $\Sigma_p$ is $\varkappa(\delta)$-close to a $k$-fold spherical suspension $S^k(\Sigma)$, where $\Sigma$ is a space of curvature $\ge1$.
Notice that $\Sigma$ is nonempty since $k<n=\dim M$ (see Remark \ref{rem:susp}).
Therefore, there exists $q\in M$ near $p$ such that $|g(p)g(q)|<\varkappa(\delta)|pq|$, where $g$ denotes the distance coordinate associated with the strainer $\{(a_i,b_i)\}_{i=1}^k$.
Let $q_j\in M_j$ be a sequence converging to $q$.
Then $|g_j(p_j)g_j(q_j)|<\varkappa(\delta)|p_jq_j|$ for large $j$, where $g_j$ denotes the distance coordinate associated with the strainer $\{(a_i^j,b_i^j)\}_{i=1}^k$.
By the inequality \eqref{eq:reg}, we have $|f_j(p_j)f_j(q_j)|<\varkappa(\delta)|p_jq_j|$.
Hence, by the $(1-\varkappa(\delta))$-openness, we can find $\hat q_j\in f_j^{-1}(f_j(p_j))$ such that $|\hat q_jq_j|<\varkappa(\delta)|p_jq_j|$.
In particular, we have $|p_j\hat q_j|>|pq|/2$ if $\delta$ is small enough.
This contradicts the assumption that $\diam f_j^{-1}(f_j(p_j))\cap B(p_j,\ell\delta)\to 0$.
\end{proof}

\begin{proof}[Proof of Proposition \ref{prop:low}]
Let $q$ be as in Lemma \ref{lem:low} and let $r$ be the midpoint of a shortest path connecting $p$ and $q$.
Put $a_{k+1}:=p$ and $b_{k+1}:=q$.
Then, by the inequality \eqref{eq:reg} and Lemma \ref{lem:str}, we see that $\{(a_i,b_i)\}_{i=1}^{k+1}$ is a $(k+1,\varkappa(\delta))$-strainer at $r$ with length $>\rho/2$.
Furthermore, we have $|f(p)f(r)|<\varkappa(\delta)|pr|$.
Thus, by the $(1-\varkappa(\delta))$-openness, we can find $s\in f^{-1}(f(p))$ such that $|ps|<\varkappa(\delta)|pr|$.
Then, $\{(a_i,b_i)\}_{i=1}^{k+1}$ is also a $(k+1,\varkappa(\delta))$-strainer for $s$ with length $>\rho/3$.
Hence, $(f,|a_{k+1}\cdot|)$ is a $\varkappa(\delta)$-almost regular map around $s$.
Repeating this argument $(n-k)$-times, we get $\hat p\in f^{-1}(f(p))$ and $h=(|a_{k+1}\cdot|,\dots,|a_n\cdot|)$ such that $(f,h)$ is a $\varkappa(\delta)$-almost regular map associated with a $(n,\varkappa(\delta))$-strainer at  $\hat p$ with length $>\hat\ell=c(D,v,\ell)$.
Thus, by Proposition \ref{prop:reg}(2),  $(f,h)$ is a $\varkappa(\delta)$-almost isometry from $B(\hat p,\hat\ell\delta)$ to an open subset of $\mathbb R^n$.
Therefore, the restriction of $h$ to $f^{-1}(f(p))$ gives a $\varkappa(\delta)$-almost isometry from a neighborhood of $\hat p$ in $f^{-1}(f(p))$ to an $(\hat\ell\delta/2)$-ball in $\mathbb R^{n-k}$.
This completes the proof (fix small $\delta$ such that the last $\varkappa(\delta)$ is less than $1/2$).
\end{proof}

\begin{rem}\label{rem:recti}
Let $f:B(p,\ell)\to\mathbb R^k$ be a $\delta$-almost regular map associated with a $(k,\delta)$-strainer at $p$ with length $>\ell$.
Set $F:=f^{-1}(f(p))\cap B(p,\ell\delta)$.
Suppose that $x\in F$ has an $(n-k,\delta)$-strainer $\{(a_{k+i},b_{k+i})\}_{i=1}^{n-k}$ such that $a_{k+i},b_{k+i}\in F$.
Then, the same argument as above shows that the restriction of the distance coordinate of this strainer to $F$ gives a $\varkappa(\delta)$-almost isometry from a neighborhood of $x$ in $F$ to an open subset in $\mathbb R^{n-k}$.
We can also prove that the complement of the set of all such points $x$ in $F$ has Hausdorff dimension at most $n-k-1$.
The proof is similar to that of \cite[10,6]{BGP}.
See also Lemma \ref{lem:conti2}.
\end{rem}

\subsection{Upper bound for the volume of the fibers}\label{sec:up}

In this section, we prove the right inequality of Theorem \ref{thm:fbr}(2).
The proof is based on the theory of noncritical maps by Perelman \cite{P1} and the rescaling technique for collapsing sequences by Yamaguchi \cite{Y3}.
We always assume that $\delta$ is much smaller than $\varepsilon$ and $c(\varepsilon)$ and that so is $\varkappa(\delta)$ (see \S\ref{sec:note}).

We consider the following regularity of maps in this section:

\begin{dfn}\label{dfn:nc'}
Let $f:M\to\mathbb R^k$ be a map defined on an Alexandrov space $M$.
For positive numbers $\varepsilon$, $\delta$ and $\rho$, we say that $f$ is \textit{$(\varepsilon,\delta,\rho)$-noncritical} at $p\in M$ if there exists a map $g=(|a_1\cdot|,\dots,|a_k\cdot|):M\to\mathbb R^k$, where $a_i\in M$, such that
\begin{enumerate}
\item $|(f(x)-f(y))-(g(x)-g(y))|<\delta|xy|$ for any $x,y\in B(p,\rho)$;
\item $|a_ip|>\rho$ and $\tilde\angle a_ipa_{i'}>\pi/2-\delta$ for all $i\neq i'$;
\item there exists $w\in M$ such that $|wp|>\rho$ and $\tilde\angle a_ipw>\pi/2+\varepsilon$ for all $i$.
\end{enumerate}
\end{dfn}

\begin{rem}\label{rem:nc'}
This definition is a special case of Definition \ref{dfn:nc} (except for the existence of $\rho$).
In particular, if $f$ is $(\varepsilon,\delta,\rho)$-noncritical at $p$, then $k\le n$ and it is $c(\varepsilon)$-open on $B(p,\rho\delta)$ (see Proposition \ref{prop:open}).
\end{rem}

\begin{rem}\label{rem:ncreg}
Let $f:B(p,\ell)\to\mathbb R^k$ be a $\delta$-almost regular map associated with a $(k,\delta)$-strainer $\{(a_i,b_i)\}_{i=1}^k$ at $p$ with length $>\ell$.
Then, it is $(c,\delta,c\ell)$-noncritical at $p$.
Indeed, we can find $w\in M$ such that $|wp|>c\ell$ and $\tilde\angle a_ipw>\pi/2+c$ by Proposition \ref{prop:str}.
\end{rem}

For a (compact) metric space $X$, we denote by $\beta_\nu(X)$ the maximal possible number of $\nu$-discrete points in $X$.
Furthermore, for simplicity, we introduce a notation
\[v_m(X):=\sup_{0<\nu<1}\nu^m\beta_\nu(X)\]
for $m\ge0$.
Clearly, $\vol_m(X)\le C(m)\cdot v_m(X)$, where $\vol_m$ denotes the $m$-dimensional Hausdorff measure and $C(m)$ is a constant depending only on $m$.

We prove the following proposition in this section:

\begin{prop}\label{prop:up}
Let $M$ be an $n$-dimensional Alexandrov space and $k\le n$.
For a map $f:M\to\mathbb R^k$ and $p\in M$, set $F:=f^{-1}(f(p))$.
Let $F(\varepsilon,\delta,\rho)$ be the set of all $(\varepsilon,\delta,\rho)$-noncritical points of $f$ in $F$.
Then, for any sufficiently small $\delta>0$ (depending only on $n$, $\kappa$, $D$, $\varepsilon$ and $\rho$), we have
\[v_{n-k}(F(\varepsilon,\delta,\rho)\cap B(p,D))<C(D,\varepsilon,\rho).\]
In particular,
\[\vol_{n-k}(F(\varepsilon,\delta,\rho)\cap B(p,D))<C(D,\varepsilon,\rho).\]
\end{prop}

\begin{proof}[Proof of the right inequality of Theorem \ref{thm:fbr}(2)]
Let $f:M\to X$ be the map of Theorem \ref{thm:con} and let $p$, $\hat p$, $\varphi$, $\hat \varphi$ be as in Theorem \ref{thm:con}.
Then, by the inequality \eqref{eq:con} and Remark \ref{rem:ncreg}, $\varphi\circ f$ is $(c,\varkappa(\delta),c\ell\delta^2/2)$-noncritical on $B(\hat p,\ell\delta^2/2)$.
Furthermore, the diameter of the fiber $f^{-1}(p)$ is less than $C\mu$, where $\mu\ll\ell\delta^2$ (see Remark \ref{rem:fbr}).
In particular, every point of $f^{-1}(p)$ is a $(c,\varkappa(\delta),c\ell\delta^2/2)$-noncritical point of $\varphi\circ f$.
Thus, rescaling $M$ by the reciprocal of the diameter of $f^{-1}(p)$ and applying Proposition \ref{prop:up}, we obtain the desired estimate.
\end{proof}

\begin{proof}[Proof of Proposition \ref{prop:up}]
We prove it by reverse induction on $k$.
First, consider the base case $k=n$.
The following argument is the same as the second half of the proof of Proposition \ref{prop:open}.
Let $x\in F(\varepsilon,\delta,\rho)$ and let $a_i,w\in M$ be as in Definition \ref{dfn:nc'}.
Suppose there exists $y\in B(x,\rho\delta)$ such that $f(x)=f(y)$ and $x\neq y$.
If $|wx|\ge|wy|$, then by the definition of $(\varepsilon,\delta,\rho)$-noncriticality, we have
\begin{gather*}
\tilde\angle a_iya_{i'}>\pi/2-\varkappa(\delta),\quad\tilde\angle a_iyw>\pi/2+c(\varepsilon),\\
\tilde\angle a_iyx>\pi/2-\varkappa(\delta),\quad\tilde\angle wyx>\pi/2-\varkappa(\delta)
\end{gather*}
for all $1\le i\neq i'\le n$.
This contradicts Lemma \ref{lem:cnt} for $\Sigma_y$.
We also get a contradiction when $|wx|\le|wy|$.
Therefore, the set $F(\varepsilon,\delta,\rho)$ is $\rho\delta$-discrete.
In particular, the cardinality of $F(\varepsilon,\delta,\rho)\cap B(p,D)$ is bounded above by some constant $C(D,\varepsilon, \rho)$ (note that we can take $\delta=c(\varepsilon)$ so that Lemma \ref{lem:cnt} holds in the above argument).

Next, consider the case $k<n$.
We argue by contradiction.
Thus, it is sufficient to prove the following:

\begin{prop}\label{prop:up'}
Suppose a sequence $(M_j,p_j)$ of $n$-dimensional Alexandrov spaces converges to an $l$-dimensional Alexandrov space $(X,p)$ in the pointed Gromov-Hausdorff topology.
Let $f_j:M_j\to\mathbb R^k$ and $F_j:=f_j^{-1}(f_j(p_j))$.
Then, for any positive numbers $D,\varepsilon,\rho$ and $\delta_j\to0$, we have
\[\liminf_{j\to\infty}v_{n-k}(F_j(\varepsilon,\delta_j,\rho)\cap B(p_j,D))<\infty,\]
where $F_j(\varepsilon,\delta_j,\rho)$ denotes the set of all $(\varepsilon,\delta_j,\rho)$-noncritical points of $f_j$ in $F_j$.
\end{prop}

We prove it by reverse induction on $l$, the dimension of the limit space.
Note that $k\le l\le n$ (the left inequality follows from Lemma \ref{lem:cnt}; in particular, $l\ge1$).
By the compactness of the limit set of $F_j(\varepsilon,\delta_j,\rho)\cap B(p_j,D)$, the proof is reduced to the following local version:

\begin{lem}\label{lem:up}
Suppose $x_j\in F(\varepsilon,\delta_j,\rho)\cap B(p_j,D)$ converges to $x\in\bar B(p,D)$.
Then, there exists $d>0$ (independent of $j$) such that
\[\liminf_{j\to\infty}v_{n-k}(F_j(\varepsilon,\delta_j,\rho)\cap B(x_j,d))<\infty.\]
\end{lem}

Let $x_j$ and $x$ be as above.
Since it is sufficient to prove Proposition \ref{prop:up'} for smaller $\varepsilon$, we may assume $\vol_{l-1}\Sigma_x\ge\varepsilon$.
By the $(\varepsilon,\delta_j,\rho)$-noncriticality of $f_j$ at $x_j$, there exist $g_j=(|a_1^j\cdot|,\dots,|a_k^j\cdot|):M_j\to\mathbb R^k$ and $w_j\in M_j$ satisfying the conditions of Definition \ref{dfn:nc'}.
Passing to a subsequence, we may assume that $w_j$ converges to $w\in X$.
The following argument is similar to Setting \ref{set:per}.
Let $0<\omega<\delta$, where $\delta\ll\varepsilon$ is different from $\delta_j$.
Since $\vol_{l-1}\Sigma_x\ge\varepsilon$, there exists an $\omega$-discrete set $\{\xi_\alpha\}_{\alpha=1}^N$ in the $\varepsilon/100$-neighborhood of $w'_x\in \Sigma_x$ such that $N\ge L/\omega^{l-1}$, where $L=c(\varepsilon)$.
Let $w_\alpha\in M$ be a point near $x$ in the direction $\xi_\alpha$ and $w_\alpha^j\in M_j$ a lift of $w_\alpha$ (i.e.\ $w_\alpha^j\to w_\alpha$).
Then, we have
\[\tilde\angle w_\alpha xw_{\alpha'}>\omega,\quad\liminf_{j\to\infty}\tilde\angle a_i^jx_jw_\alpha^j>\pi/2+\varepsilon/2\]
for all $1\le\alpha\neq\alpha'\le N$ and $1\le i\le k$ (note that $|a_i^jx_j|$ may go to infinity in the second inequality).
Let $0<r<\delta\min_\alpha\{\rho,|xw_\alpha|\}$ be sufficiently small.
Then, for any $y\in B(x,2r)$ and $y_j\in B(x_j,2r)$, we have
\[\tilde\angle w_\alpha yw_{\alpha'}>\omega,\quad\liminf_{j\to\infty}\tilde\angle a_i^jy_jw_\alpha^j>\pi/2+\varepsilon/3\]
for all $1\le\alpha\neq\alpha'\le N$ and $1\le i\le k$.
Define $\sigma:X\to\mathbb R$ by
\[\sigma:=\frac1 N\sum_{\alpha=1}^N|w_\alpha\cdot|.\]

The maps $f_j$ are uniformly Lipschitz on $B(x_j,\rho)$ by Definition \ref{dfn:nc'}(1).
Hence, we may assume that the normalized map $f_j(\cdot)-f_j(x_j)$ converges to some $f:B(x,\rho)\to\mathbb R^k$.
Note that $g_j(\cdot)-g_j(x_j)$ also converges to $f$ by Definition \ref{dfn:nc'}(1) since $\delta_j\to0$.
Set
\[F^+:=\{y\in B(x,\rho)\mid f^i(y)\ge f^i(x)=0,\ 1\le i\le k\},\]
where $f^i$ denotes the $i$-th component of $f$.

\begin{clm}[cf.\ {\cite[3.4 Assertion 1]{P1}}]\label{clm:up1}
One of the following holds:
\begin{enumerate}
\item There exist $a_x\in X$ and $\rho_x>0$ such that $(f_j,|a_x^j\cdot|)$ is $(c(\varepsilon),\varkappa(\delta),\rho_x)$-noncritical at $x_j$ for sufficiently large $j$, where $a_x^j\in M_j$ is a lift of $a_x$.
\item The restriction of $\sigma$ to $F^+\cap\bar B(x,r)$ has a strict maximum value at $x$.
More precisely, we have
\[\sigma(x)\ge\sigma(y)+c(\varepsilon)|xy|\]
for any $y\in F^+\cap\bar B(x,r)$.
\end{enumerate}
\end{clm}

\begin{proof}
The proof is the same as that of \cite[3.4 Assertion 1]{P1} (see Lemma \ref{lem:per1}).
Take $a=c(\varepsilon)$ and $b=c(\varepsilon)$ such that $a\ll b\ll L$.
Suppose that (2) does not hold for this $a$.
Let $y\in F^+\cap\bar B(x,r)$ be such that $\sigma(x)<\sigma(y)+a|xy|$.
Then, we have $\frac1 N\sum_{\alpha=1}^N\cos\angle w_\alpha xy<2a$ (where $\angle w_\alpha xy$ is the minimum angle between $xw_\alpha$ and $xy$).
On the other hand, $N\ge L/\omega^{l-1}$ and the number of $\alpha$ such that $|\angle w_\alpha xy-\pi/2|\le b$ is less than $Cb/\omega^{l-1}$ by Lemma \ref{lem:vol}.
Hence, there exists $\alpha$ such that $\angle w_\alpha xy>\pi/2+b$; otherwise, $\frac1 N\sum_{\alpha=1}^N\cos\angle w_\alpha xy>c(b)>2a$, a contradiction.
Choose a point $w_x$ on a shortest path $xw_\alpha$ so close to $x$ that $\tilde\angle w_xxy>\pi/2+b$.
Take a lift $w_x^j$ of $w_x$ on a shortest path $x_jw_\alpha^j$ (we may assume $x_jw_\alpha^j$ converges to $xw_\alpha$).
Set $a_x:=y$ and let $a_x^j\in M_j$ be a lift of $a_x$.
Then, we have
\[\tilde\angle a_i^jx_ja_x^j>\pi/2-\varkappa(\delta),\quad\tilde\angle a_x^jx_jw_x^j>\pi/2+b\]
for all $1\le i\le k$ and sufficiently large $j$ (the first inequality follows from $a_x\in F^+$).
The other conditions of noncriticality of $(f_j,|a_x^j\cdot|)$ at $x_j$ are obviously satisfied.
\end{proof}

We first prove Lemma \ref{lem:up} in the case of Claim \ref{clm:up1}(1).
Recall that Proposition \ref{prop:up} for $k+1$ holds by the induction hypothesis.
Let $0<d<\rho_x\delta$.
Then, $(f_j,|a_x^j\cdot|)$ is $(c(\varepsilon),\varkappa(\delta),\rho_x/2)$-noncritical on $B(x_j,2d)$.
In particular, it is $c(\varepsilon)$-open on $B(x_j,2d)$.
Given a $\nu$-discrete set in $F_j\cap B(x_j,d)$, split it into $2d/\delta\nu$ classes so that the difference of the distances from $a_x^j$ to any two points in the same class is no more than $\delta\nu$.
Then, for each class, we can take a corresponding $\nu/2$-discrete set in a fiber of $(f_j,|a_x^j\cdot|)$ by the $c(\varepsilon)$-openness.
Proposition \ref{prop:up} for $k+1$ implies that the number of such $\nu/2$-discrete points is less than $C(\varepsilon)\nu^{-(n-k-1)}$ (we have to apply Proposition \ref{prop:up} to the rescaled space $\rho_x^{-1}M_j$ so that the choice of $\delta$ depends only on $\varepsilon$ because $\rho_x$ depends on $\delta$).
Thus, Lemma \ref{lem:up} follows.

From now on, we consider the case of Claim \ref{clm:up1}(2).

\begin{sbclm}[cf.\ {\cite[3.4 Assertion 2]{P1}}]\label{sbclm:up1}
In the case of Claim \ref{clm:up1}(2), for any $y\in F^+\cap B(x,r)\setminus\{x\}$, there exists $\alpha$ such that $\tilde\angle w_\alpha yx>\pi/2+c(\varepsilon)$.
\end{sbclm}

\begin{proof}
Take $a=c(\varepsilon)\ll L$.
It suffices to show that $\angle w_\alpha xy\le\pi/2-a$ for some $\alpha$ (where $\angle w_\alpha xy$ is the minimum angle).
Suppose that $\angle w_\alpha xy>\pi/2-a$ for all $\alpha$.
Then, by Lemma \ref{lem:vol}, the number of $\alpha$ such that $\angle w_\alpha xy\le\pi/2+a$ is less than $Ca/\omega^{l-1}$.
Since $N\ge L/\omega^{l-1}$ and $a\ll L$, we have
\[\sigma'(y'_x)\ge\frac1 N\sum_{\alpha=1}^N-\cos\angle w_\alpha xy>c(a).\]

Recall that $f$ is the limit of $g_j(\cdot)-g_j(x_j)$.
Since $g_j^i$, the $i$-th component of $g_j$, are uniformly $\lambda$-concave near $x_j$, so is $f^i$ near $x$ (where $\lambda>0$ depends only on $\kappa$ and $\rho$).
Hence, for any point $z$ (close to $x$) on a shortest path $xy$, we have
\[f^i(z)\ge f^i(x)-\frac\lambda2|xz||yz|\ge f^i(x)-\frac{\lambda r}2|xz|\]
since $y\in F^+$.
Furthermore, $f$ is $c(\varepsilon)$-open near $x$ since it is a limit of $c(\varepsilon)$-maps.
Therefore, we can find a point $\hat z\in F^+\cap B(x,r)$ such that $c(\varepsilon)|\hat zz|\le(\lambda r/2)|xz|$.
Since $x$ is the maximum point of $\sigma$ on $F^+\cap\bar B(x,r)$, we have
\[\sigma(x)\ge\sigma(\hat z)\ge\sigma(z)-|\hat zz|\ge\sigma(z)-c(\varepsilon)^{-1}\frac{\lambda r}2|xz|.\]
Since $z$ is arbitrary on $xy$, we obtain $\sigma'(y'_x)\le c(\varepsilon)^{-1}\lambda r/2$.
This contradicts $\sigma'(y'_x)>c(a)$ because $r<\rho\delta\ll c(\varepsilon)$.
\end{proof}

Define $\sigma_j:M_j\to\mathbb R$ by
\[\sigma_j:=\frac1 N\sum_{\alpha=1}^N|w_\alpha^j\cdot|\]
and set
\[F_j^+:=\{y_j\in B(x_j,\rho)\mid f_j^i(y_j)\ge f_j^i(x_j),\ 1\le i\le k\},\]
where $f_j^i$ denotes the $i$-th component of $f_j$.
Let $\hat x_j$ be a maximum point of $\sigma_j$ on $F_j^+\cap\bar B(x_j,r)$.
Then, Claim \ref{clm:up1}(2) implies that $\hat x_j$ converges to $x$.

\begin{sbclm}[cf.\ {\cite[3.9]{P2}}]\label{sbclm:up2}
Indeed, $\hat x_j\in F_j$ for large $j$.
\end{sbclm}

\begin{proof}
Fix large $j$ and suppose the contrary.
Then, $f_j^{i_0}(\hat x_j)>f_j^{i_0}(x_j)$ for some $i_0$.
Let $A_i'$ (resp.\ $W'$) be the set of all directions of shortest paths from $\hat x_j$ to $a_i^j$ (resp.\ $w_\alpha^j$ for all $\alpha$).
Then, by Lemma \ref{lem:dir}(1), there exists $\xi\in\Sigma_{\hat x_j}$ such that 
\[\angle(\xi,A_{i_0}')<\pi/2-c(\varepsilon),\quad\angle(\xi,A_i')=\pi/2,\quad\angle(\xi,W')>\pi/2+c(\varepsilon)\]
for all $i\neq i_0$.
Let $y\in M_j$ be a point near $\hat x_j$ in the direction $\xi$ such that
\[f_j^{i_0}(y)>f_j^{i_0}(x_j),\quad|g_j^i(y)-g_j^i(\hat x_j)|<\delta|y\hat x_j|,\quad\sigma_j(y)>\sigma_j(\hat x_j)+c(\varepsilon)|y\hat x_j|\]
for all $i\neq i_0$.
In particular, the middle inequality above and the inequality of Definition \ref{dfn:nc'}(1) imply that
\[f_j^i(y)\ge f_j^i(\hat x_j)-(\delta+\delta_j)|y\hat x_j|\ge f_j^i(x_j)-2\delta|y\hat x_j|\]
for all $i\neq i_0$.
Therefore, by the $c(\varepsilon)$-openness of $f_j$, we can find a point $\hat y\in F_j^+\cap B(x_j,r)$ such that $c(\varepsilon)|\hat yy|\le2\delta|y\hat x_j|$.
Then, we have
\[\sigma_j(\hat y)\ge\sigma_j(y)-|\hat yy|\ge\sigma_j(\hat x_j)+(c(\varepsilon)-2\delta c(\varepsilon)^{-1})|y\hat x_j|\]
This contradicts the choice of $\hat x_j$ because $\delta\ll c(\varepsilon)$.
\end{proof}

The above observation is important, but actually will not be used below.
Now, we prove the following rescaling theorem (compare Subclaim \ref{sbclm:up1}):

\begin{clm}[cf.\ {\cite[3.2]{Y3}}]\label{clm:up2}
In the above situation, one of the following holds:
\begin{enumerate}
\item There exists a subsequence of $\{j\}$ such that for any $y_j\in F_j^+\cap B(\hat x_j,r)\setminus\{\hat x_j\}$, there is $\alpha$ such that $\tilde\angle w_\alpha^jy_j\hat x_j>\pi/2+c(\varepsilon)$.
\item There exists a sequence $d_j\to 0$ of positive numbers such that
\begin{enumerate}
\item for any $y_j\in F_j^+\cap B(\hat x_j,r)\setminus\bar B(\hat x_j,d_j)$, there is $\alpha$ such that $\tilde\angle w_\alpha^jy_j\hat x_j>\pi/2+c(\varepsilon)$;
\item for any limit $(\hat X,\hat x)$ of the rescaled spaces $(\frac1{d_j}M_j,\hat x_j)$, we have $\dim\hat X\ge l+1$, where $l=\dim X$.
\end{enumerate}
\end{enumerate}
In particular, if $l=n$, then (1) holds.
Note that when (1) (resp.\ (2)) holds, $(f_j,|\hat x_j\cdot|)$ is $(c(\varepsilon),\varkappa(\delta),s/2)$-noncritical on $F_j\cap\partial B(\hat x_j,s)$ for any $s\in(0,r)$ (resp.\ $s\in(d_j,r)$).
\end{clm}

\begin{proof}
Let $a_1=c(\varepsilon)$ be the constant in Subclaim \ref{sbclm:up1}.
Take $a=c(\varepsilon)$ such that $a\ll L$ and $a\le a_1$.
Suppose that (1) does not hold for this $a$.
Then, for any large $j$, there exists $y_j\in F_j^+\cap B(\hat x_j,r)\setminus\{\hat x_j\}$ such that $\tilde\angle w_\alpha^jy_j\hat x_j\le\pi/2+a$ for all $\alpha$.
Let $\hat y_j$ be a farthest point from $\hat x_j$ among such $y_j$ and set $d_j:=|\hat x_j\hat y_j|$.
Then, (2)(i) is trivial.
Moreover, Subclaim \ref{sbclm:up1} implies that $d_j\to0$ since $a\le a_1$.

Now, we prove (2)(ii).
Suppose $(\frac1{d_j}M_j,\hat x_j)$ converges to an Alexandrov space $(\hat X,\hat x)$ of nonnegative curvature.
Passing to a subsequence, we may assume that $\hat y_j$ converges to $\hat y\in\hat X$.
We may further assume that shortest paths $\hat x_j\hat y_j$ and $\hat x_jw_\alpha^j$ converge to a shortest path $\hat x\hat y$ and a ray $\gamma_\alpha$ starting from $\hat x$, respectively.
Let $v_j,v_\alpha^j\in\Sigma_{\hat x_j}$ and $v,v_\alpha\in\Sigma_{\hat x}$ denote the directions of them, respectively.
Then, by the monotonicity of angles, we have
\begin{gather*}
\angle(v_\alpha,v_{\alpha'})\ge\lim_{j\to\infty}\tilde\angle w_\alpha^j\hat x_jw_{\alpha'}^j\ge\omega,\\
\angle(v,v_\alpha)\ge\lim_{j\to\infty}\tilde\angle\hat y_j\hat x_jw_\alpha^j=\pi-\lim_{j\to\infty}\tilde\angle\hat x_j\hat y_jw_\alpha^j\ge\pi/2-a
\end{gather*}
for all $1\le\alpha\neq\alpha'\le N$.
Suppose (2)(ii) does not hold, i.e.\ $\dim\hat X=l$.
Then, an argument similar to the first paragraph of the proof of Subclaim \ref{sbclm:up1} implies that
\begin{equation}\label{eq:up1}
\frac1 N\sum_{\alpha=1}^N-\cos\angle(v,v_\alpha)>c(a).
\end{equation}

On the other hand, by an argument similar to the second paragraph of the proof of Subclaim \ref{sbclm:up1}, we show that
\begin{equation}\label{eq:up2}
\sigma_j'(v_j)\le c(\varepsilon)^{-1}\left(\frac{\lambda d_j}2+2\delta_j\right),
\end{equation}
where $\lambda>0$ is a constant depending only on $\kappa$ and $\rho$ such that $g_j^i$ are uniformly $\lambda$-concave near $x_j$.
Fix large $j$ and let $z$ be an arbitrary point on a shortest path $\hat x_j\hat y_j$.
Then, by the $\lambda$-concavity of $g_j^i$, we have
\[\frac{|\hat y_jz|}{d_j}(g_j^i(z)-g_j^i(\hat x_j))+\frac{|\hat x_jz|}{d_j}(g_j^i(z)-g_j^i(\hat y_j))\ge-\frac\lambda2|\hat x_jz||\hat y_jz|.\]
Moreover, since $\hat x_j,\hat y_j\in F_j^+$, by the inequality of Definition \ref{dfn:nc'}(1), we have
\begin{gather*}
f_j^i(z)-f_j^i(x_j)\ge
\begin{cases}
f_j^i(z)-f_j^i(\hat x_j)\ge(g_j^i(z)-g_j^i(\hat x_j))-\delta_j|\hat x_jz|,\\
\hfil f_j^i(z)-f_j^i(\hat y_j)\ge(g_j^i(z)-g_j^i(\hat y_j))-\delta_j|\hat y_jz|.
\end{cases}
\end{gather*}
Combining the three inequalities above, we obtain
\begin{align*}
f_j^i(z)-f_j^i(x_j)&\ge-\frac\lambda2|\hat x_jz||\hat y_jz|-2\delta_j\frac{|\hat x_jz||\hat y_jz|}{d_j}\\
&\ge-\left(\frac{\lambda d_j}2+2\delta_j\right)|\hat x_jz|.
\end{align*}
Furthermore, since $f_j$ is $c(\varepsilon)$-open near $x_j$, we can find a point $\hat z\in F_j^+\cap B(x_j,r)$ such that $c(\varepsilon)|\hat zz|\le(\lambda d_j/2+2\delta_j)|\hat x_jz|$.
Since $\hat x_j$ is the maximum point of $\sigma_j$ on $F_j^+\cap\bar B(x_j,r)$, we have
\[\sigma_j(\hat x_j)\ge\sigma(\hat z)\ge\sigma(z)-|\hat zz|\ge\sigma(z)-c(\varepsilon)^{-1}\left(\frac{\lambda d_j}2+2\delta_j\right)|\hat x_jz|.\]
Since $z$ is arbitrary on $\hat x_j\hat y_j$, we obtain the desired inequality \eqref{eq:up2}.

Therefore, applying the first variation formula to the left-hand side of the inequality \eqref{eq:up2} and passing to the limit, we obtain
\[\frac1 N\sum_{\alpha=1}^N-\cos\angle(v,v_\alpha)\le0\]
by the lower semicontinuity of angles (choose $v_\alpha^j$ to be closest to $v_j$ among all directions from $\hat x_j$ to $w_\alpha^j$ so that the first variation formula holds).
This contradicts the inequality \eqref{eq:up1}.
\end{proof}

Now, we give the proof of Lemma \ref{lem:up} in the case of Claim \ref{clm:up1}(2).
Recall that Proposition \ref{prop:up} for $k+1$ and Proposition \ref{prop:up'} for $k$ and $l'\ge l+1$ hold by the induction hypothesis.
The proof is divided into the two cases of Claim \ref{clm:up2}.

First, suppose Claim \ref{clm:up2}(1) holds.
In this case, Lemma \ref{lem:up} follows from Proposition \ref{prop:up'} for $k+1$ as in the case of Claim \ref{clm:up1}(1).
Consider a $\nu$-discrete set in $F_j\cap B(\hat x_j,r)\setminus B(\hat x_j,\nu/2)$.
Note that $(f_j,|\hat x_j\cdot|)$ is $(c(\varepsilon),\varkappa(\delta),\nu/3)$-noncritical (and hence $c(\varepsilon)$-open) on a small neighborhood of $F_j\cap B(\hat x_j,r)\setminus B(\hat x_j,\nu/2)$.
Split the $\nu$-discrete set into $r/\delta\nu$ classes so that the difference of the distances from $\hat x_j$ to any two points in the same class is no more than $\delta\nu$.
Then, for each class, we can take a corresponding $\nu/2$-discrete set of $F_j\cap\partial B(\hat x_j,s)$ for some $s\in(\nu/2,r)$ by the $c(\varepsilon)$-openness.
Proposition \ref{prop:up} for $k+1$ implies that the number of such $\nu/2$-discrete points is less than $C(\varepsilon)\nu^{-(n-k-1)}$ (apply Proposition \ref{prop:up} to the rescaled space $s^{-1}M$).
Thus, Lemma \ref{lem:up} follows.
In particular, Proposition \ref{prop:up'} for $l=n$ holds (the base case of the reverse induction on $l$).

Next, suppose Claim \ref{clm:up2}(2) holds.
Then, the same argument as in the previous case shows that $v_{n-k}(F_j\cap B(\hat x_j,r)\setminus\bar B(\hat x_j,d_j))$ is uniformly bounded above.
On the other hand, Proposition \ref{prop:up'} for $l'\ge l+1$ applied to $\frac1{d_j}M_j$ implies that $v_{n-k}(F_j(\varepsilon,\delta_j,d_j)\cap\bar B(\hat x_j,d_j))$ is uniformly bounded above for some subsequence.
This completes the proof.
\end{proof}

\subsection{Almost continuity of the volume of the fibers}\label{sec:conti}

In this section, we prove Theorem \ref{thm:fbr}(3).
The proof is similar to that of the continuity of the volume of Alexandrov spaces under the Gromov-Hausdorff convergence (\cite[10.8]{BGP}, \cite[3.5]{S}, \cite[0.6]{Y2}).
By Remarks \ref{rem:fbr} and \ref{rem:regcon}, it suffices to prove the following:

\begin{prop}\label{prop:conti}
Let $M$ be an $n$-dimensional Alexandrov space, $p\in M$ and $k<n$.
Let  $f:B(p,\ell)\to\mathbb R^k$ be a $\delta$-almost regular map associated with a $(k,\delta)$-strainer $\{(a_i,b_i)\}_{i=1}^k$ at $p$ with length $>\ell$.
For simplicity, assume that $f^{-1}(\bar B(f(p),r))$ is contained in $B(p,\ell\delta)$ for some $0<r\ll\ell\delta$.
Set $F_v:=f^{-1}(v)$ for $v\in B(f(p),r)$.
Then, for any sufficiently close $u,v\in B(f(p),r)$, we have
\[\left|\frac{\vol_{n-k}F_u}{\vol_{n-k}F_v}-1\right|<\varkappa(\delta).\]
\end{prop}

Let $F_v$ be as above.
Let $m$ be a positive integer and $\theta$, $\rho$ positive numbers.
We denote by $F_v(m,\theta,\rho)$ the set of all points in $F_v$ having an $(m,\theta)$-strainer $\{(a_{k+i},b_{k+i})\}_{i=1}^m$ with length $>\rho$ such that $a_{k+i},b_{k+i}\in F_v$.
The above proposition immediately follows from the following two lemmas:

\begin{lem}\label{lem:conti1}
For any small $\rho>0$, if $u,v\in B(f(p),r)$ are sufficiently close, then there exists a $\varkappa(\delta)$-almost isometry from $F_v(n-k,\delta,\rho)$ to an open subset of $F_u$.
\end{lem}

\begin{proof}
Since the proof is similar to that of Theorem \ref{thm:con} in the noncollapsing case, we only outline it.
Assume $\rho\ll|f^{-1}(\bar B(f(p),r)),\partial B(p,\ell\delta)|$ and $|uv|\ll\rho\delta$.
Fix $x\in F_v(n-k,\delta,\rho)$ and let $\{(a_{k+i},b_{k+i})\}_{i=1}^{n-k}$ be an $(n-k,\delta)$-strainer at $x$ with length $>\rho$ such that $a_{k+i},b_{k+i}\in F_v$.
Then, the inequality \eqref{eq:reg} implies that $\{(a_i,b_i)\}_{i=1}^n$ is an $(n,\varkappa(\delta))$-strainer at $x$ with length $>\rho$.
Hence, by Proposition \ref{prop:reg}(1), there exists a $\varkappa(\delta)$-almost isometry from $B(x,\rho\delta)$ to an open subset of $\mathbb R^n$ whose first $k$ components coincide with $f$.
Therefore, a translation with respect to the $f$-coordinate gives a $\varkappa(\delta)$-almost isometry from $B(x,\rho\delta/2)\cap F_v$ to an open subset of $F_u$, provided $|uv|\ll\rho\delta$.
Gluing such local maps as in the proof of Theorem \ref{thm:con}, we obtain a global almost isometry.
\end{proof}

\begin{lem}[cf.\ {\cite[10.9]{BGP}}]\label{lem:conti2}
Let $1\le m\le n-k$ and $\theta,\rho>0$.
Then, for any $0<\nu<\rho$, we have
\[\nu^{n-k}\beta_\nu(F_v\setminus F_v(m,\theta,\rho))<\varkappa_\theta(\rho),\]
where $\varkappa_\theta(\rho)$ is a positive function depending only on $n$, $\kappa$ and $\theta$ such that $\varkappa_\theta(\rho)\to0$ as $\rho\to0$.
In particular,
\[\vol_{n-k}(F_v\setminus F_v(n-k,\delta,\rho))<\varkappa_\delta(\rho).\]
\end{lem}

\begin{proof}
The proof is carried out by the induction on $m$ as in \cite[10.9]{BGP}.
Let $\rho_1\gg\rho$ and $\theta_1=c(\theta)\ll\theta$.
Take a maximal $(\rho_1\theta_1)$-discrete net $\{x_\alpha\}_{\alpha=1}^N$ of $F_v(m-1,\theta_1,\rho_1)$.
Then, we have $N<C(\rho_1\theta_1)^{-(n-k)}$ by Proposition \ref{prop:up}.
Fix $\alpha$ and let $\{(a_{k+i},b_{k+i})\}_{i=1}^{m-1}$ be an $(m-1,\theta_1)$-strainer at $x_\alpha$ with length $>\rho_1$ such that $a_{k+i},b_{k+i}\in F_v$.
Divide $B(x_\alpha,\rho_1\theta_1)\cap F_v$ into $(2\rho_1/\rho)^{m-1}$ classes $\{D_{\alpha\beta}\}_{\beta}$ so that if $x$, $y$ belong to the same class, then we have $||a_{k+i}x|-|a_{k+i}y||\le\rho\theta_1$ for all $1\le i\le m-1$.
By the argument in \cite[10.5]{BGP}, if there exist sufficiently many $\rho$-discrete points (depending on $\theta$) in $D_{\alpha\beta}$, then we can choose three of them so that they form a $(1,\theta)$-strainer (and a strained point) with length $>\rho$.
Therefore, the number of $\rho$-discrete points in $D_{\alpha\beta}\setminus F_v(m,\theta,\rho)$ is less than $C(\theta)$.
Hence, we can cover $F_v(m-1,\theta_1,\rho_1)\setminus F_v(m,\theta,\rho)$ by at most $C(\theta)\rho_1^{-(n-k)}(\rho_1/\rho)^{m-1}$ balls of radius $\rho$ (recall $\theta_1=c(\theta)$).
By Proposition \ref{prop:up} again, the number of $\nu$-discrete points in each $\rho$-ball is less than $C(\rho/\nu)^{n-k}$ for any $0<\nu<\rho$.
Thus, we have
\[\nu^{n-k}\beta_\nu(F_v(m-1,\theta_1,\rho_1)\setminus F_v(m,\theta,\rho))<C(\theta)\left(\frac\rho{\rho_1}\right)^{n-k-m+1}\]
for any $0<\nu<\rho$.
Together with the induction hypothesis and a suitable choice of $\rho_1$, this yields the desired estimate.
\end{proof}

\begin{proof}[Proof of Proposition \ref{prop:conti}]
By Proposition \ref{prop:low} and Remark \ref{rem:low}, we may assume $\vol_{n-k}F_v>c(p)$ for any $v\in B(f(p),r)$.
Let $\rho>0$ be so small that $\varkappa_\delta(\rho)\ll c(p)$.
Then, by the above two lemmas, for sufficiently close $u,v\in B(f(p),r)$, we have
\begin{align*}
(1+\varkappa(\delta))\vol_{n-k}F_u&\ge\vol_{n-k}F_v(n-k,\delta,\rho)\\
&\ge\vol_{n-k}F_v-\varkappa_\delta(\rho)\ge(1-\varkappa(\delta))\vol_{n-k}F_v.
\end{align*}
This completes the proof.
\end{proof}

\section*{Acknowledgment}

The author would like to thank Prof.\ Takao Yamaguchi for his advice and encouragement.
He is also grateful to the referee for the careful reading and useful comments.

\end{document}